\documentclass[onefignum,onetabnum]{siamart190516}
\usepackage[utf8]{inputenc}
\usepackage{mathtools}
\usepackage{dsfont}
\usepackage{amsmath}
\usepackage{amssymb}
\usepackage{caption}
\usepackage{subcaption}
\usepackage{algorithmicx}
\usepackage{algorithm}
\usepackage{algpseudocode}
\usepackage{comment}

\title{Compound Krylov subspace methods for parametric linear systems}
\author{Antti Autio and Antti Hannukainen} 
\date{Spring 2021}

\usepackage{bm}
\renewcommand{\vec}{\bm}

\newsiamthm{remark}{Remark}
\usepackage{color}

\begin{document}

\maketitle

\begin{abstract}
\noindent
In this work, we propose a reduced basis method for efficient solution of parametric linear systems. The coefficient matrix is assumed to be a linear matrix-valued function that is symmetric and positive definite for admissible values of the parameter $\vec{\sigma}\in \mathbb{R}^s$. We propose a solution strategy where one first computes a basis for the appropriate compound Krylov subspace and then uses this basis to compute a subspace solution for multiple $\vec{\sigma}$. Three kinds of compound Krylov subspaces are discussed. Error estimate is given for the subspace solution from each of these spaces. Theoretical results are demonstrated by numerical examples related to solving parameter dependent elliptic PDEs using the finite element method (FEM).
\end{abstract}



\section{Introduction}

Denote by $\mathbb{S}^n_{++} \subset \mathbb{R}^{n\times n}$ the set of real, symmetric, and positive definite (s.p.d.) $n\times n$ -- matrices. Let $\vec{b} \in \mathbb{R}^n$, \emph{parameter set} $S\subset \mathbb{R}^s$, and $A:\mathbb{R}^s \mapsto \mathbb{R}^{n\times n}$ be a \emph{linear matrix-valued function} such that $A(S) \subset \mathbb{S}^n_{++}$. This work concerns efficient solution of the linear system: find $\vec{x}(\vec{\sigma}) \in \mathbb{R}^{n}$ satisfying
\begin{equation}
    A(\vec{\sigma})\vec{x}(\vec{\sigma}) = \vec{b}
    \label{eq:linsys}
\end{equation}
for multiple values of the parameter $\vec{\sigma} \in S$. We call $\vec{x}: S\mapsto \mathbb{R}^n$ defined point-wise by \eqref{eq:linsys} as the \emph{parameter-to-solution map}. 

Our motivation for studying \eqref{eq:linsys} arises from the solution of elliptic PDEs with spatially varying coefficient functions using the finite element method (FEM), see Section~\ref{sec:fem}. As parameter dependent PDEs are related to several interesting engineering problems, their solution has attracted lots of attention. Research has been done both before and after spatial discretization. Parametric PDEs have especially been studied in the context of uncertainty quantification, where the parameter $\vec{\sigma}$ is typically related to a truncated Polynomial Chaos or Karhunen-Loève expansion of a random coefficient field, see \cite{BaTeZo:04,BaNoTe:07,ScGi:11}. 

Currently, there exist three main approaches for the solution of \eqref{eq:linsys} or the underlying parametric PDE. One can approximate the parameter-to-solution map $\vec{x}(\vec{\sigma})$ using a Galerkin method in the parameter space, see \cite{BaTeZo:04,ScGi:11}. These methods often combine discretization of spatial and parameter dimensions. Second alternative is to apply a collocation method, where $\vec{x}(\vec{\sigma})$  is first evaluated at collocation points and then approximated by interpolation, \cite{BaNoTe:07}. To break the curse of dimensionality, several sparse and adaptive families of collocation points have been proposed \cite{BaCoDa:17}. Finally, one can construct a reduced basis or a subspace of $\mathbb{R}^n$ that can accurately represent the solution $\vec{x}(\vec{\sigma})$ for desired $\vec{\sigma} \in S$ \cite{QuMaNe:16}. The reduced basis is constructed by evaluating the solution at sampling points that are selected, e.g., by a greedy algorithm \cite{BuMaPaPr:12}. The solution $\vec{x}(\vec{\sigma})$ is then approximated point-wise by computing subspace solution using the reduced basis.

In this work, we propose a reduced basis method for the solution of \eqref{eq:linsys} that is inspired by the Conjugate Gradient (CG) method. The solution $\vec{x}(\vec{\sigma})$ of the linear system \eqref{eq:linsys} for a single $\vec{\sigma} \in S$ can be approximated efficiently using CG if the condition number of $A(\vec{\sigma})$ is close to one. The CG method is an iteration for finding a sequence of approximate solutions to linear systems with s.p.d.~coefficient matrices, see \cite{HeSt:1953} and \cite{axelsson:1994}. Each CG-iterate is the subspace solution from the \emph{Krylov subspace} corresponding to the linear system and the iteration index. The $j$th Krylov subspace related to the model problem~\eqref{eq:linsys} is defined as
\begin{equation}
K_j(A(\vec{\sigma}),\vec{b}) := \mathop{span}\{ \vec{b},A(\vec{\sigma}) \vec{b},\ldots, A(\vec{\sigma})^{j-1} \vec{b} \} \mbox{ for any $j \in \mathbb{N}$}.
\end{equation}
Observe that $K_j(A(\vec{\sigma}),\vec{b})$ is dependent on $\vec{\sigma}$.

The convergence of CG is well studied; the estimated number of iterations required to compute an approximate solution with desired error grows with the condition number. The condition numbers of coefficient matrices related to the FE-solution of elliptic PDEs are large and increase when the applied finite element mesh is refined. Hence, if the CG method is used in this setting, a preconditioner is required to improve convergence. 

In this work, we propose one kind of exact and two kinds of approximate \emph{compound Krylov} (CK) subspace methods to efficiently compute approximate solutions to \eqref{eq:linsys} for multiple $\vec{\sigma} \in S$. The computation proceeds in two stages:
\begin{enumerate}
    \item {\bf Off-line stage:} Compute a basis for the applied compound Krylov subspace. 
    \item {\bf On-line stage:} Use the basis constructed in the off-line stage to compute subspace solution to \eqref{eq:linsys} for multiple $\vec{\sigma} \in S$.
\end{enumerate}
\noindent In practice, the computational cost related to the first stage is large whereas computing the subspace solution for a favorable $A(\vec{\sigma})$ is fast. Therefore our proposed method is most beneficial when the parameter-to-solution map is evaluated for a large number of parameter vectors $\vec{\sigma}$.

The family of exact compound Krylov subspaces $CK_j(A,\vec{b}) \subset \mathbb{R}^n$ is designed to satisfy the inclusion
\begin{equation}
\label{eq:inclusion}
    K_j(A(\vec{\sigma}),\vec{b}) \subset CK_j(A,\vec{b}) \quad \mbox{for all $\vec{\sigma} \in S$ and $j\in \mathbb{N}$}.
\end{equation}
Observe that $CK_j(A,\vec{b})$ is independent of $\vec{\sigma}$ but dependent on $\vec{b}$ as well as on the matrix-valued function $A$. Due to the inclusion in \eqref{eq:inclusion} and the \emph{best approximation} property of subspace methods, the subspace solution to \eqref{eq:linsys} from $CK_j(A,\vec{b})$ is at least as accurate as the $j$th iterate produced by the CG method for any $\vec{\sigma} \in S$.

Constructing a subspace satisfying the inclusion~\eqref{eq:inclusion} requires treating the $\vec{\sigma}$--dependency of the linear matrix-valued function $A$. We use the linearity of $A$ and define $CK_j(A,\vec{b})$ as the union of subspaces containing the range of the mapping 
\begin{equation}
\label{eq:map}
    \vec{\sigma} \rightarrow A(\vec{\sigma})^k \vec{b} \quad \mbox{for} \quad k\in \{1,\ldots,j\}.
\end{equation}
Particularly, we reformulate the terms $A(\vec{\sigma}) \vec{b}$ as $A(\vec{\sigma})\vec{b} = L(\vec{b}) \vec{\sigma}$, where $L(\vec{b}) \in \mathbb{R}^{n \times s}$ is a \emph{linearisation matrix} independent of $\vec{\sigma}$. This reformulation allows us to easily compute the range of $L(\vec{b})$ that contains $A(\vec{\sigma}) \vec{b}$ for any $\vec{\sigma} \in S$. Such  \emph{linearisation process} can be repeated for terms $A(\vec{\sigma})^k \vec{b}$, and thus, to find a basis for the compound Krylov subspace $CK_j(A,\vec{b})$. Special care must be taken to cope with the \emph{exponentially growing column dimension} of the linearisation matrices.  The dimension and the computational cost are reduced by two kinds of \emph{approximate compound Krylov subspaces} that are defined by including low rank approximations to the linearisation process.

The proposed CK-solvers are subspace methods, and as such they produce the best possible approximation to the exact solution $\vec{x}(\vec{\sigma})$ from the method subspace in the norm associated with the coefficient matrix $A(\vec{\sigma})$. We take advantage of this property in error analysis. Particularly, we show that both kinds of approximate CK-method subspaces approximately contain a solution candidate appearing in the error analysis of the Conjugate Gradient (CG) method for any $\vec{\sigma} \in S$. Our final error estimate guarantees that the CK-solutions have a comparable error with the CG method if sufficiently accurate low rank approximations of the linearisation matrices are used.  

This work is organised as follows. In Section~\ref{sec:bg} we give a brief review of subspace and Conjugate Gradient (CG) methods and discuss how Problem \eqref{eq:linsys} is related to the FE-solution of the Poisson's equation with varying material data or geometry. Compound Krylov subspaces are discussed in Section~\ref{sec:pKrylov}. Implementation of the CK method is outlined in Section~\ref{sec:practical}. The proposed methods and analytical results are  illustrated in Section~\ref{sec:num} by numerical examples.  We conclude with a discussion of the obtained results and future work in Section~\ref{sec:con}

\section{Background} 
\label{sec:bg}
In this section, we first discuss linear matrix-valued functions and their representation. Then we give two examples of linear systems of the type \eqref{eq:linsys} that are related to finite element solution of parametric PDEs. Finally, we briefly review subspace methods and CG error analysis that are a prerequisite for  Section~\ref{sec:pKrylov}.  
\subsection{Linear matrix-valued functions}
Linear matrix-valued functions are defined as usual:
\begin{definition} Function $F:\mathbb{R}^s \mapsto \mathbb{R}^{n \times n}$ is called linear if for any $\alpha \in \mathbb{R}$ and $\vec{\sigma}_1,\vec{\sigma}_2 \in \mathbb{R}^s$
\begin{equation*}
F(\vec{\sigma}_1 + \vec{\sigma}_2) = F(\vec{\sigma}_1) + F(\vec{\sigma}_2) 
\quad 
\mbox{and} \quad 
F(\alpha \vec{\sigma}_1) = \alpha F(\vec{\sigma}_1). 
\end{equation*}
\end{definition}
Naturally, any linear matrix-valued function $F:\mathbb{R}^s \mapsto \mathbb{R}^{n \times n}$ admits the representation $F(\vec{\sigma}) = \sum_{i=1}^s \sigma_i F_i $ for $\{F_i\}_{i=1}^s \subset \mathbb{R}^{n\times n}$ independent of $\vec{\sigma}$. Particularly, there exists $\{A_i\}_{i=1}^s \subset \mathbb{R}^{n\times n}$ independent of $\vec{\sigma}$ such that 
\begin{equation}
\label{eq:sum_rep}
A(\vec{\sigma}) = \sum_{i=1}^s \sigma_i A_i \quad \mbox{for any} \quad  \vec{\sigma} \in \mathbb{R}^s.
\end{equation}
\subsection{Application in FEM}
\label{sec:fem}

The motivation for our proposed methods comes from solving parameter dependent partial differential equations using the finite element method. Let domain $\Omega \subset \mathbb{R}^d$, where the dimension is $d=2$ or $d=3$, have sufficiently regular boundary and $f\in L^2(\Omega)$. Consider the weak form of the \emph{modified Poisson's equation}: find $u_\sigma \in H_0^1(\Omega)$ satisfying
\begin{equation}
    \int_\Omega \nabla u_\sigma \cdot C(\vec{\sigma},x) \nabla v = \int_\Omega f v \quad \text{for all } v\in H_0^1(\Omega).
    \label{eq:weakform}
\end{equation}
We are interested in solving this problem multiple times with different values of $\vec{\sigma}$ using finite elements. Assume the function $C: \mathbb{R}^s \times \mathbb{R}^d \mapsto \mathbb{R}^{d\times d}$ is in the form 
\begin{equation}
    C(\vec{\sigma},x) = \sum_{i=1}^s \psi_i(x) \sigma_i,
    \label{eq:cmatrix}
\end{equation}
where $\psi_i \in L^{\infty}(\Omega; \mathbb{R}^{d\times d})$. The function $C$ is chosen so that $C(\vec{\sigma},x) \in \mathbb{S}_{++}^{d}$ for a.e. $x \in \Omega$ and all $\vec{\sigma}\in S$. In this case, the Lax-Milgram lemma guarantees the existence of a unique solution to \eqref{eq:weakform} for any $\vec{\sigma} \in S$, see, e.g. \cite{evans:1998}.

The weak problem \eqref{eq:weakform} is solved using FEM by limiting it to some finite element space $V_{FE} \subset H^1_0(\Omega)$. This is, one solves the problem: find $u_{\sigma,FE}\in V_{FE}$ satisfying 
\begin{equation}
    \int_\Omega \nabla u_{\sigma,FE} \cdot C(\vec{\sigma},x) \nabla v = \int_\Omega f v \quad \mbox{for all $v\in V_{FE}$}. 
    \label{eq:weakFE}
\end{equation}
The FE-space is finite dimensional and has a basis $\{\phi_i\}_{i=1}^n$. The basis functions $\phi_i$ are defined with the help of a \emph{mesh}, a partition of the domain $\Omega$ to subdomains called elements. The maximum diameter of these elements is called the \emph{mesh size}. For an introduction on FEM, see \cite{Braess:2007,brenner_scott:1994}.

The problem \eqref{eq:weakFE} is equivalent to the matrix equation: find $\hat{\vec{x}}(\vec{\sigma}) \in \mathbb{R}^n$ satisfying $K(\vec{\sigma})\vec{\hat{x}}(\vec{\sigma})=\vec{\hat{b}}$, where $K:\mathbb{R}^s \mapsto \mathbb{R}^{n\times n}$ and $\vec{b} \in \mathbb{R}^n$ are defined as 
\begin{equation}
    K(\vec{\sigma})_{i j} = \int_\Omega \nabla \phi_i \cdot C(\vec{\sigma},x) \nabla \phi_j \quad \mbox{and} \quad \vec{\hat{b}}_i =  \int_\Omega f \phi_i \quad \mbox{for $i,j \in \{1,\ldots,n\}$}.
\label{eq:kmatrix}
\end{equation} 
%
%
The accuracy of the approximate solution produced by the compound Krylov solver depends, among other things, on the condition number of the coefficient matrix, see Theorems~\ref{thm:error1} and~\ref{thm:error2}. If the condition number is close to one, these methods converge rapidly. However, the condition number of $K(\vec{\sigma})$ defined in \eqref{eq:kmatrix} grows with decreasing mesh size and is typically much larger than one leading to slow convergence of CK-methods, see \cite[Chapter B.6]{ToWi:05} for analysis in the case $C(\vec{\sigma},x)=I$. To speed up convergence we improve conditioning by applying a split preconditioner. Let $\bar{K}\in\mathbb{R}^{n\times n}$ satisfy
\begin{equation}
    \bar{K}_{i j} = \int \nabla \phi_i \cdot \bar{C}(x) \nabla \phi_j \quad \mbox{for} \quad i,j\in \{1,\ldots,n\},
    \label{eq:kbar}
\end{equation}
where the function $\bar{C}:\mathbb{R}^d \mapsto \mathbb{R}^{d\times d}$ satisfies the following: there exists $\alpha,\beta\in \mathbb{R}^+$ such that
\begin{equation*}
    \alpha \vec{\eta}^T \bar{C}(x) \vec{\eta} \leq \vec{\eta}^T C(\vec{\sigma},x) \vec{\eta} \leq \beta \vec{\eta}^T \bar{C}(x) \vec{\eta} 
\end{equation*}
for all $\vec{\eta} \in \mathbb{R}^d$ and for a.e.~$x\in \Omega$. In addition, let $R$ be the Cholesky factor of $\bar{K}$ so that $\bar{K}=R R^T$. We can now write the linear system as
\begin{equation*}
    R^{-1} K (\vec{\sigma})R^{-T} R^{T} \vec{\hat{x}}(\vec{\sigma}) = R^{-1}\vec{\hat{b}} .
\end{equation*}
Redefining $R^{-1}K(\vec{\sigma})R^{-T}=A(\vec{\sigma})$, $\vec{x}(\vec{\sigma})=R^{T}\vec{\hat{x}}(\vec{\sigma})$, and $\vec{b}=R^{-1}\vec{\hat{b}}$ yields the preconditioned linear system: find $\vec{x}(\vec{\sigma})\in \mathbb{R}^n$ satisfying
\begin{equation}
    A(\vec{\sigma})\vec{x}(\vec{\sigma})=\vec{b}.
    \label{eq:lineareq}
\end{equation}
The assumption \eqref{eq:cmatrix} ensures that the matrix $A$ is a linear matrix-valued function. This is, we have arrived to an instance of \eqref{eq:linsys}. Next, we give an estimate for the condition number of the coefficient matrix $A(\vec{\sigma})$ in~\eqref{eq:lineareq}.
\begin{lemma}
\label{lemma:cmatrix}
Let $C:\mathbb{R}^s \times \mathbb{R}^d \rightarrow \mathbb{R}^{d\times d}$ and $\bar{C}: \mathbb{R}^d \rightarrow \mathbb{R}^{d\times d}$. Let $K:\mathbb{R}^s \mapsto \mathbb{R}^{n\times n}$ be as defined in \eqref{eq:kmatrix}, $\bar{K} \in \mathbb{R}^{n\times n}$ as defined in \eqref{eq:kbar}, $R$ the Cholesky factor of $\bar{K}$, and $A(\vec{\sigma})=R^{-1}K(\vec{\sigma})R^{-T}$. Assume that there exists $\alpha,\beta\in \mathbb{R}^+$ such that
\begin{equation}
    \label{eq:cmatrix_ass}
    \alpha \vec{\eta}^T \bar{C}(x)\vec{\eta} \leq \vec{\eta}^T C(\vec{\sigma},x) \vec{\eta} \leq \beta \vec{\eta}^T \bar{C}(x) \vec{\eta} 
\end{equation}
for a.e. $x \in \Omega$ and every $\vec{\eta} \in \mathbb{R}^s$. Then it holds for any $\vec{\sigma} \in S$ that $\kappa_2(A(\vec{\sigma})) \leq \beta/\alpha$, where $\kappa_2(A(\vec{\sigma}))$ is the condition number of $A(\vec{\sigma})$ in the Euclidean norm.
\end{lemma}
\begin{proof}
We use the \textit{Rayleigh quotient}. According to it the smallest and largest eigenvalues of the matrix $A(\vec{\sigma})$ are
\begin{equation}
    \lambda_{\min} =
    \min_{\vec{y} \in \mathbb{R}^n} \frac{\vec{y}^T R^{-1}K(\vec{\sigma})R^{-T}\vec{y}}{\vec{y}^T \vec{y}}, \quad  \lambda_{\max} =
    \max_{\vec{y} \in \mathbb{R}^n} \frac{\vec{y}^T R^{-1}K(\vec{\sigma})R^{-T}\vec{y}}{\vec{y}^T \vec{y}}.
    \label{eq:lambdaminmax}
\end{equation}
We apply a change of variables $\vec{\hat{y}}=R^{-T}\vec{y}$ and write the quotient as
\begin{equation*}
    \frac{\vec{\hat{y}}^T K(\vec{\sigma})\vec{\hat{y}}}{\vec{\hat{y}}^T \bar{K} \vec{\hat{y}}}.
\end{equation*}
Let $v_{\vec{\hat{y}}} \in V_{FE}$ be defined as $v_{\vec{\hat{y}}} = \sum_{i=1}^n \hat{y}_i \phi_i$ for $\vec{\hat{y}} \in \mathbb{R}^n$. By \eqref{eq:kmatrix}, the quotient can be written as:
\begin{equation*}
    \frac{\int \nabla v_{\vec{\hat{y}}} \cdot C(\vec{\sigma},x) \nabla v_{\vec{\hat{y}}}}{\int \nabla v_{\vec{\hat{y}}} \cdot \bar{C}(x) \nabla v_{\vec{\hat{y}}}}.
\end{equation*}
Using the assumptions in \eqref{eq:cmatrix_ass} we find that:
\begin{align*}
    \alpha \leq \frac{\int \nabla v_{\vec{\hat{y}}} \cdot C(\vec{\sigma},x) \nabla v_{\vec{\hat{y}}}}{\int \nabla v_{\vec{\hat{y}}} \cdot \bar{C}(x) \nabla v_{\vec{\hat{y}}}} \leq \beta \quad \mbox{for any $\vec{\hat{y}} \in \mathbb{R}^n$.}
\end{align*}
Combining this with \eqref{eq:lambdaminmax} yields estimate for the smallest and the largest eigenvalue of $A(\vec{\sigma})$
\begin{equation*}
    \alpha \leq \lambda_{\min}(A(\vec{\sigma})) \quad \mbox{and} \quad \lambda_{\max}(A(\vec{\sigma})) \leq \beta
\end{equation*}
for any $\vec{\sigma} \in S$. Recalling the definition of the condition number completes the proof.
\end{proof}

\subsubsection{Example 1: Piecewise constant material parameter}
\label{sec:matparam}
Let the subdomains $\{\Omega_i\}_{i=1}^s$ be non-overlapping and satisfy $\bigcup_{i=1}^s \overline{ \Omega_i} = \overline{\Omega}$. Let $a\in \mathbb{R}$, $a>1$ and $S = \{ \vec{\sigma} \in \mathbb{R}^s \; | \; \sigma_i \in [1,a] \quad \mbox{for $i \in \{1,\ldots,s\}$} \; \}$. We consider solving the problem: find $u_\sigma \in H^1_0(\Omega)$ satisfying
\begin{equation}
\label{eq:cond_eq}
    \sum_i \int_{\Omega_i} \sigma_i \nabla u_\sigma \cdot \nabla v = \int_{\Omega} f v \quad \text{for all } v\in H_0^1(\Omega)
\end{equation}
for multiple $\vec{\sigma} \in S$. The parameter $\sigma_i$ may physically correspond, for example, to the electrical conductivity in the subdomain $\Omega_i$. Eq. \eqref{eq:cond_eq} is an instance of the abstract problem in \eqref{eq:weakFE} with
\begin{equation}
\label{eq:ex1_C} C(\vec{\sigma},x) = \sum_{i=1}^s \sigma_i I \mathcal{X}_{\Omega_i},
\end{equation}
where $\mathcal{X}_{\omega}:\mathbb{R}^d \mapsto \{0,1\}$ is the characteristic function of the set $\omega \subset \mathbb{R}^d$.
\begin{figure}[h]
    \centering
    \includegraphics[width=0.5\textwidth]{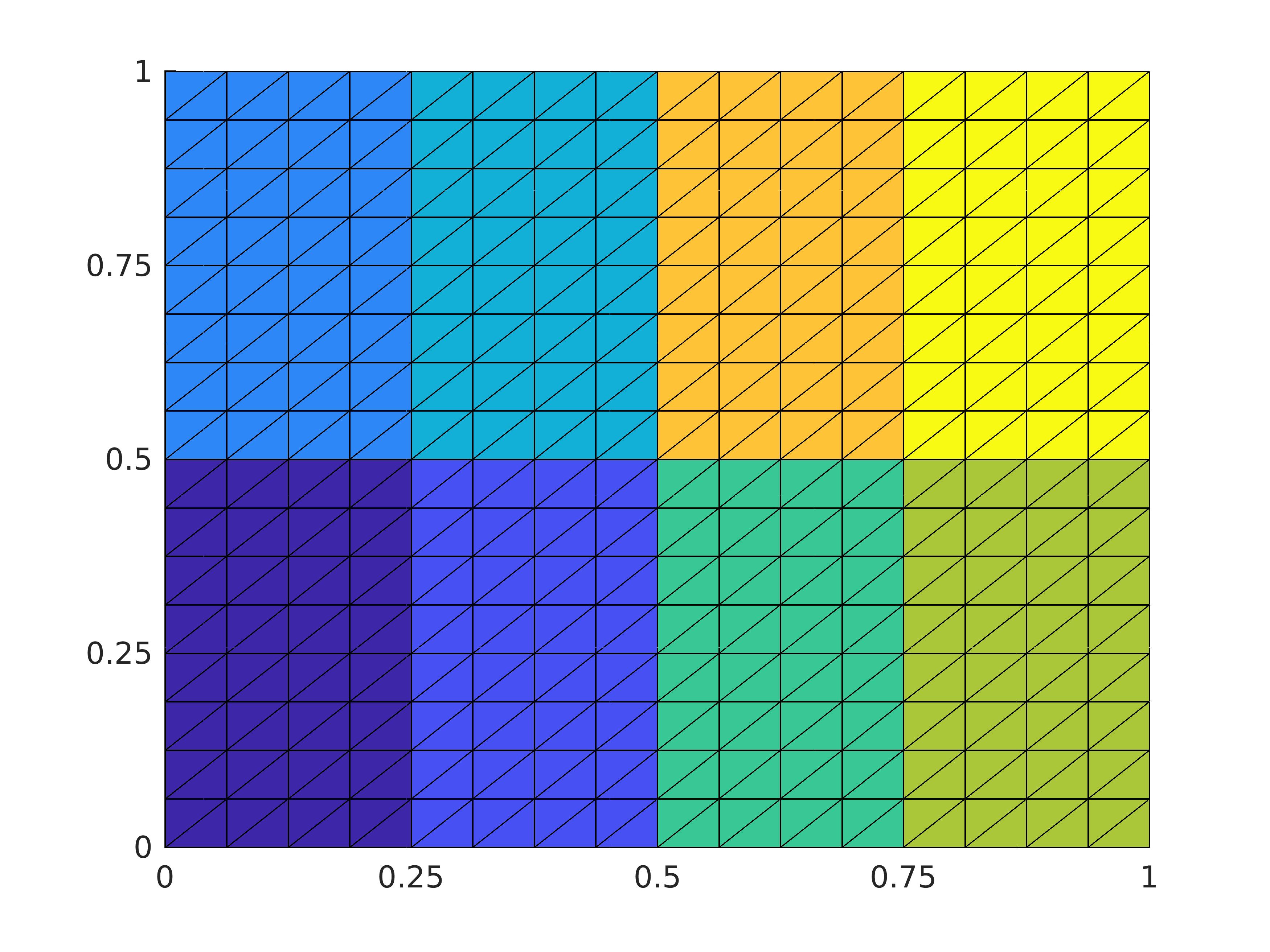}
    \caption{Example of a $2\times 4$ - checkerboard pattern. The subdomains $\{\Omega_i\}_{i=1}^8$, marked with different colors, have independent material parameters. The depicted FE-mesh conforms with the subdomain interfaces.}
    \label{fig:checkerboard}
\end{figure}
As a demonstration we use $N\times M$ - checkerboard patterns where the domain $(0,1)$ is divided into $NM$ subdomains by slicing it into $N$ horizontal and $M$ vertical strips, see Fig. \ref{fig:checkerboard}.  We choose the preconditioning coefficient matrix $\bar{C}=I$. Estimate for the condition number of $A:\mathbb{R}^s\mapsto\mathbb{R}^{n\times n}$ corresponding to \eqref{eq:cond_eq} and $\overline{C}=I$ follows using Lemma \ref{lemma:cmatrix}. As $\sigma_i \in [1,a]$, we have 
\begin{equation*}
    \vec{\eta}^T \begin{pmatrix} 1&0\\0&1 \end{pmatrix} \vec{\eta}
    \leq \vec{\eta}^T C(\vec{\sigma},x) \vec{\eta} 
    \leq \vec{\eta}^T \begin{pmatrix} a&0\\0&a \end{pmatrix} \vec{\eta}
    \quad \mbox{for a.e. $x \in \Omega$ and any $\eta \in \mathbb{R}^d$}.
\end{equation*}
%
Therefore $\alpha=1$ and $\beta=a$ in \eqref{eq:cmatrix_ass}, and $\kappa_2(A(\vec{\sigma})) \leq a$.

\subsubsection{Example 2: Deformation of geometry}
\label{sec:hole}
A slightly more complicated example of \eqref{eq:linsys} is a problem where the parameter $\vec{\sigma}$ is related to deformation of the domain. As an example we solve the Poisson's equation in a rectangular domain with a spherical hole in multiple positions along the y-direction. Let $r = 0.15$, $\vec{r}_0 = \begin{bmatrix} 0.5 & 0.5 \end{bmatrix}^T$, and $\Omega_l = (0,1)^2 \setminus B(\vec{r}_0 + \vec{e}_2 l, r)$. Consider solving the problem: find $u_l \in H^1_0(\Omega_l)$ satisfying
\begin{equation}
    \int_{\Omega_l} \nabla u_l \cdot \nabla v_l = \int_{\Omega_l} v_l \quad \text{for all } v_l\in H_0^1(\Omega_l)
    \label{eq:poissonweak}
\end{equation}
for multiple $l\in (-a,a)$, where $a\in \mathbb{R}^+$, $a < 1/3$. We reduced the problem~\eqref{eq:poissonweak} to the reference domain $\widehat{\Omega} = (0,1)^2\setminus B(\vec{r}_0, r)$ by using the coordinate transformation $F_l: \widehat{\Omega} \rightarrow \Omega_l$ defined as
\begin{equation}
    F_l(\hat{\vec{x}}) = \hat{\vec{x}} + \vec{e}_2 \tau(\hat{x}_2)l .
    \label{eq:tranformation}
\end{equation}
Here $l$ specifies the length of the translation and $\tau$ is a piecewise linear function defined as follows:
\begin{equation*}
    \tau(t) = 
    \begin{cases}3t & t\in [0, \frac{1}{3})\\
    1 & t\in[\frac{1}{3},\frac{2}{3}]\\
    -3t+3 & t\in(\frac{2}{3},1]\end{cases}.
\end{equation*}
\begin{figure}[h]
    \centering
    \includegraphics[width=0.9\textwidth]{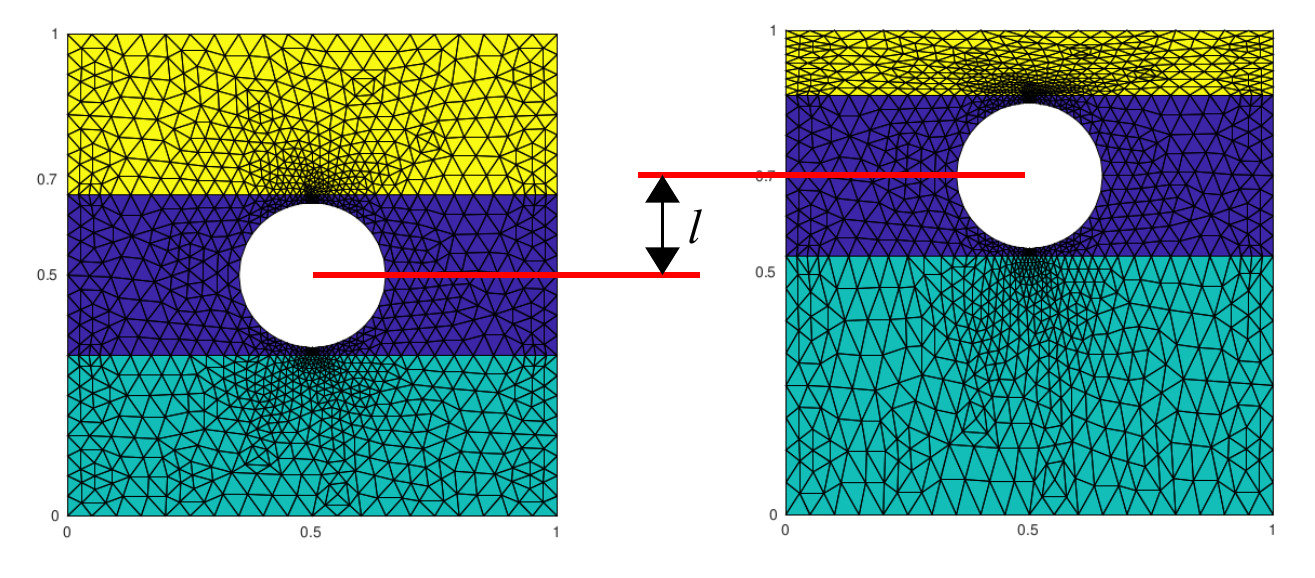}
    \caption{The geometry of the reference domain $\widehat{\Omega}$ (left) and the transformed domain $\Omega_l$ (right). Here $l=0.2$.}
    \label{fig:rect_transf}
\end{figure}
The domain before and after the transformation can be seen in Figure \ref{fig:rect_transf}. For convenience, we name the three subdomains with different transformation rules from bottom to top as $\widehat{\Omega}_{1}$, $\widehat{\Omega}_{2}$ and $\widehat{\Omega}_{3}$. The Jacobian of the transformation is 
\begin{equation*}
    D F_l(\hat{\vec{x}}) = \begin{pmatrix}1&0\\0&1+\tau'(\hat{x}_2)l\end{pmatrix} 
\end{equation*}
Applying the change of variables $\vec{x} = F_l(\hat{\vec{x}})$ in~\eqref{eq:poissonweak} and defining $\hat{u}_l(\hat{\vec{x}}):=u_l(F_l(\hat{\vec{x}}))$ yields the problem: find $\hat{u}_l \in H^1_0(\widehat{\Omega})$ satisfying 
\begin{equation}
\label{eq:ref_hole_prob}
    \int_{\widehat{\Omega}} \hat{\nabla} \hat{u}_l \cdot (D F_l)^{-1} (D F_l)^{-T} \det(D F_l)  \hat{\nabla} \hat{v} =
    \int_{\widehat{\Omega}} \hat{v} \det{(D F_l)} 
\end{equation}
for all $\hat{\vec{v}}_l \in H^1_0(\widehat{\Omega})$. Observe that $\det{DF_l}>0$ for $l\in(-a,a)$, hence, its absolute value can be omitted. Expanding the left hand side of \eqref{eq:ref_hole_prob} gives
\begin{equation}
\int_{\widehat{\Omega}}  \hat{v} \det{(D F_l)} = 
\int_{\widehat{\Omega}}  \hat{v} + 
l  \int_{\widehat{\Omega}_{i}}  f \hat{v} \quad \mbox{for} \quad f(\hat{\vec{x}}) = \tau^\prime(\hat{x}_2).  
\end{equation}
Hence, the problem \eqref{eq:ref_hole_prob} can be solved in two parts: find $\hat{u}_{l0}, \hat{u}_{l1} \in H^1_0(\widehat{\Omega})$  satisfying:
\begin{equation}
\label{eq:ex2_1}
    \int_{\widehat{\Omega}} \hat{\nabla} \hat{u}_{l0}\ \cdot (D F_l)^{-1} (D F_l)^{-T} \det(D F_l)  \hat{\nabla} \hat{v} =
    \int_{\widehat{\Omega}}  \hat{v} \det{(D F_l)}
\end{equation}
and 
\begin{equation}
\label{eq:ex2_2}
    \int_{\widehat{\Omega}} \hat{\nabla} \hat{u}_{l1} \cdot (D F_l)^{-1} (D F_l)^{-T} \det(D F_l)  \hat{\nabla} \hat{v} =
    \int_{\widehat{\Omega}} f \hat{v} \det{(D F_l)},
\end{equation}
for any $v \in H^1_0(\widehat{\Omega})$. Then $\hat{u}_l = \hat{u}_{l0} + l \hat{u}_{l1}$. Next, we reformulate the RHS so that both of these problems are instances of the abstract problem \eqref{eq:weakform}. The extra term induced by the change of variables to the RHS of \eqref{eq:ref_hole_prob} is:
\begin{equation*}
    (D F_l)^{-1} (D F_l)^{-T} \det(D F_l) = \begin{pmatrix}1+\tau'(\hat{x}_2)l&0\\0&(1+\tau'(\hat{x}_2)l)^{-1} \end{pmatrix} .
\end{equation*}
Since $\tau$ is a piecewise linear function, the above matrix is piecewise constant with respect to the spatial variable $x$ and depends only on $l$. Explicitly,
\begin{equation}
    \begin{pmatrix}1+3l&0\\0&(1+3l)^{-1} \end{pmatrix} \mathcal{X}_{\widehat{\Omega}_1}+
    \begin{pmatrix}1&0\\0&1 \end{pmatrix} \mathcal{X}_{\widehat{\Omega}_2}+
    \begin{pmatrix}1-3l&0\\0&(1-3l)^{-1} \end{pmatrix} \mathcal{X}_{\widehat{\Omega}_3}
    \label{eq:transmatrix}.
\end{equation}
We identify the parameters $\vec{\sigma} \in \mathbb{R}^s$ with elements in the above equation as
\begin{equation}
\label{eq:l2sigma}
\sigma_1 = 1+3l,\quad \sigma_2 = (1+3l)^{-1}, \quad \sigma_5 = 1-3l,\quad \sigma_6 = (1-3l)^{-1}.
\end{equation}
The elements $\sigma_3$ and $\sigma_4$ equal to one since the Jacobian in the subdomain $\widehat{\Omega}_{2}$ is the identity. The above relations and bound $l\in(-a,a)$ define the parameter set $S\subset \mathbb{R}^6$. The LHS of \eqref{eq:ex2_1} corresponds to 
\begin{equation}
\label{eq:ex2_final}
\int_{\widehat{\Omega}} \nabla \hat{u}_{l0} \cdot C(\vec{\sigma},x)\nabla \hat{v}
\end{equation}
for 
\begin{equation} C(\vec{\sigma},x) = \begin{pmatrix} \sigma_1 & 0 \\ 0 & \sigma_2 \end{pmatrix}\mathcal{X}_{\widehat{\Omega}_{1}}(x) 
    + I \mathcal{X}_{\widehat{\Omega}_{2}}(x) 
    + \begin{pmatrix} \sigma_5 & 0 \\ 0 & \sigma_6 \end{pmatrix}\mathcal{X}_{\widehat{\Omega}_{3}}(x).
\end{equation}
Hence, it is an instance of the abstract problem \eqref{eq:weakFE}. Same applies to \eqref{eq:ex2_2}.

 Because the parameter $l$, and also $\vec{\sigma}$, vary symmetrically around $0$, we choose the preconditioning coefficient matrix $\bar{C}=I$. We proceed to estimate the condition number of $A:\mathbb{R}^s \mapsto \mathbb{R}^{n\times n}$ corresponding to \eqref{eq:ex2_final} and $\overline{C}=I$. We obtain:
%
\begin{equation*}
(1-3a) \vec{\eta}^T\vec{\eta} \leq \vec{\eta}^T C(\vec{\sigma},x) \vec{\eta}  \leq \frac{1}{1-3a} \vec{\eta}^T\vec{\eta} \quad \mbox{for a.e. $\vec{x}\in \Omega$ and $\vec{\eta} \in \mathbb{R}^d$}. 
\end{equation*}
This is, $\alpha = 1-3a$ and $\beta = (1-3a)^{-1}$ in Eq. \eqref{eq:cmatrix_ass}. By Lemma~\ref{lemma:cmatrix} the condition number satisfies  
\begin{equation}
    \kappa_2(A(\vec{\sigma})) \leq 1/(1-3a)^2 \quad \mbox{for all $\vec{\sigma} \in S$}.
    \label{eq:hole_kappa}
\end{equation}
It's worth noting that the condition number blows up when $a$ approaches $1/3$.
\subsection{Subspace methods} 



Let $B\in \mathbb{S}^{n}_{++}$, $\vec{g} \in \mathbb{R}^n$, $V \subset \mathbb{R}^n$ be a subspace, $\{ \vec{q}_i \}_{i=1}^k$ a basis of $V$,  and $Q = \begin{bmatrix} \vec{q}_1 & \ldots & \vec{q}_k \end{bmatrix}$. A subspace method computes an approximate solution $\hat{\vec{y}} \in V$ to the linear system $B \vec{y} = \vec{g}$ by first solving the auxiliary problem: find $\vec{z} \in \mathbb{R}^k$ satisfying 
\begin{equation}
\label{eq:sub_lin_sys}
    Q^T B Q \vec{z} = Q^T \vec{g}, \quad \mbox{and then setting} \quad \hat{\vec{y}} = Q \vec{z}.
\end{equation}
We call such $\hat{\vec{y}}$ as the \emph{subspace solution from $V$}.

Any $B\in \mathbb{S}^{n}_{++}$ defines an inner product $\left<\cdot,\cdot\right>_{B}$ and the induced norm $\| \cdot \|_{B}$ in $\mathbb{R}^n$: for any $\vec{z},\vec{w} \in \mathbb{R}^n$ let
\begin{equation*}
\left< \vec{z},\vec{w} \right>_{B} := \vec{z}^T B \vec{w} \quad \mbox{and} \quad \| \vec{z} \|_{B} := \left< \vec{z},\vec{z} \right>_{B}^{1/2}.
\end{equation*}
 The subspace solution $\vec{\hat{y}}$ from $V$ is the $\left<\cdot,\cdot\right>_{B}$-orthogonal projection of $\vec{y}$ onto $V$. Thus $\vec{\hat{y}}$ depends only on $V$, not on the basis $\{ \vec{q}_i \}_{i=1}^k$. For this reason, we call $V$ as \emph{the method subspace}.
 
 The error of the subspace solution $\hat{\vec{y}}$ is measured in the $B$-norm as $\| \vec{y}-\vec{\hat{y}} \|_{B}$. Because $\vec{\hat{y}}$ is the $B$-orthogonal projection of the exact solution to $V$, the $B$-norm of the error satisfies the \emph{best approximation property}:
\begin{equation}
\label{eq:proj_error}
\| \vec{y}-\vec{\hat{y}} \|_{B} = \min_{\vec{v} \in V} \| \vec{y} - \vec{v} \|_{B}. 
\end{equation}

Let $\vec{x}(\vec{\sigma})$ be the exact solution and $\vec{\hat{x}}(\vec{\sigma})$ the subspace solution from $V$ to \eqref{eq:linsys}, respectively. Our aim is to design subspace $V$, independent of $\vec{\sigma}$, such that 
\begin{equation*}
    \| \hat{\vec{x}}(\vec{\sigma}) - \vec{x}(\vec{\sigma}) \|_{A(\vec{\sigma})} \leq tol \quad \mbox{for any $\vec{\sigma} \in S$}.
\end{equation*}
We take advantage of the best approximation property and design compound Krylov method subspaces that contain either exactly or approximately a solution candidate appearing in the CG error analysis. By the best approximation property, the error of the CK-solution is then bounded by the error of this solution candidate. CG error analysis is discussed next.
\subsection{The Conjugate Gradient Method} 
\label{sec:cg}
The CG method is an iteration for finding a sequence of approximate solutions to linear systems with s.p.d. coefficient matrices, see \cite{HeSt:1953} and \cite{axelsson:1994}.  It can be understood as a line search method for minimising the energy functional associated to the linear system to be solved, or as a method finding a sequence of subspace solutions from the family of Krylov subspaces corresponding to the linear system. 

Let $B\in \mathbb{S}^{n}_{++}$, $\vec{g} \in \mathbb{R}^n$, and consider the linear system: find $\vec{y} \in \mathbb{R}^n$ satisfying 
\begin{equation}
    \label{eq:Blinsys}
    B \vec{y} = \vec{g}.
\end{equation} 
The family of Krylov subspaces corresponding to \eqref{eq:Blinsys} is defined as 
\begin{equation*}
     K_j\left(B ,\vec{g}\right) :=  \mathop{span}\{ \vec{g}, B \vec{g}, \ldots, B^{j-1} \vec{g}\} \quad \mbox{for $j \in \mathbb{N}$}.
     \label{eq:krylov}
\end{equation*}
The CG method computes a sequence of approximate solutions $\{\hat{\vec{y}}_j \}$ to \eqref{eq:Blinsys} such that $\vec{\hat{y}}_j$ is the subspace solution from $K_j(B,\vec{g})$ for each $j \in \mathbb{N}$. It does this \emph{without} computing a basis for $K_j(B,\vec{g})$, which makes CG method very memory efficient. We proceed to outline the CG error analysis, i.e., study how the error $\vec{y}-\hat{\vec{y}}_j$ depends on $B$, $\vec{g}$, and the iteration index $j$. This material can be found, e.g., from \cite{axelsson:1994}. 

First, observe the duality between vectors in $K_j(B,\vec{g})$ and $(j-1)$-degree polynomials: each $\vec{v}_j \in K_j(B, \vec{g})$ satisfies 
\begin{equation}
\label{eq:krylov_rep}
\vec{v}_j = p_{\vec{v}_j}(B) \vec{g} \quad \mbox{for} \quad p_{\vec{v}_j} \in \mathcal{P}^{j-1}.
\end{equation}
Similarly, $p(B)\vec{g} \in K_j(B,\vec{g})$ for any $p\in \mathcal{P}^{j-1}$. 

It is well known that a bound for the error norm $\| \vec{y}-\vec{\hat{y}}_j \|_{B}$ follows from the best approximation property \eqref{eq:proj_error} and constructing an approximation to $\vec{y}$ from $K_{j}(B,\vec{g})$ by utilising properties of
the Chebychev polynomials, see \cite{axelsson:1994}. By~\eqref{eq:krylov_rep}
\begin{equation*}
\vec{y} - \vec{v}_j = \left[ I - p_{\vec{v}_j}(B) B \right] \vec{y} 
\end{equation*}
for any $\vec{v}_j \in K_j(B,\vec{g})$ and $p_{\vec{v}_j} \in \mathcal{P}^{j-1}$ satisfying $\vec{v}_j = p_{\vec{v}_j}(B)\vec{g}$. Denote the set of eigenvalues of $B$ by $\Lambda(B)$. By standard arguments, 
\begin{equation}
\label{eq:cg_error1}
\| \vec{y} - \vec{v}_j \|_{B} \leq \max_{t \in \Lambda(B) } |1-t p_{\vec{v}_j}(t)| \| \vec{y} \|_{B} \quad \mbox{for any $\vec{v}_j \in K_j(B,\vec{g})$}.  
\end{equation}

The multiplier $q_{\vec{v}_j}(t) := 1-t p_{\vec{v}_j}(t)$ in \eqref{eq:cg_error1} satisfies $q_{\vec{v}_j} \in \mathcal{M}^j$, where $\mathcal{M}^j$ is the space of degree $j$ monic polynomials. One can verify that choosing appropriate $\vec{v}_j$ yields all possible multipliers in $\mathcal{M}^j$. The CG error estimate could be constructed by finding $q_{\vec{v}_j} \in \mathcal{M}^j$ that minimises the multiplicative term $\max_{t \in \Lambda(B) } |q_{\vec{v}_j}(t)|$. As there does not exist a general solution for this optimisation problem, one instead finds $q_{\vec{v}_j}^*$ that has the minimal $L^\infty(\lambda_{\min}(B),\lambda_{\max}(B))$ norm in the set $\mathcal{M}^j$ by translating and scaling the $j$th Chebychev polynomial $T_j(t)$ on $(-1,1)$ as 
\begin{equation}
\label{eq:cg_cheb}
q_{j}^*(t) = \frac{T_j(\frac{\lambda_{\max}+\lambda_{\min}-2t}{\lambda_{\max}-\lambda_{\min}})}{T_j(\frac{\lambda_{\max}+\lambda_{\min}}{\lambda_{\max}-\lambda_{\min}})} = \sum_{k=0}^j \gamma_{jk} t^k.
\end{equation}
Note that $\gamma_{j0} = q_{j}^*(0) = 1$ for any $j\in \mathbb{N}$. Let $\kappa(B)$ be the condition number of $B$, i.e., $\kappa(B)=\| B\|_2 \|B^{-1} \|_2$. Using the properties of Chebychev polynomials gives the identity
\begin{equation}
    \label{eq:chebest}
  \max_{t \in \Lambda(B) } |1-t p_{\vec{v}_j}(t)| =   2\Bigg{(}\frac{\sqrt{\kappa(B)}-1}{\sqrt{\kappa(B)}+1}\Bigg{)}^j.
\end{equation}
The element $\vec{v}^*_j \in K_j(B,\vec{g})$ satisfying $q_j^*(t) = 1-t p_{\vec{v}_j^*}(t)$ is
\begin{equation}
\label{eq:vstar}
\vec{v}^*_j = (q_{j}^*(B)-I)B^{-1} \vec{g} = \sum_{k=1}^j \gamma_{jk} B^{k-1} \vec{g} .
\end{equation}
Choosing $\vec{v}_j = \vec{v}_j^*$ in \eqref{eq:cg_error1}, using the best approximation property, and \eqref{eq:chebest} yields the error bound 
\begin{equation}
  \| \vec{y}-\vec{\hat{y}}_j \|_{B} \leq 2\Bigg{(}\frac{\sqrt{\kappa(B)}-1}{\sqrt{\kappa(B)}+1}\Bigg{)}^j \| \vec{y} \|_{B} \quad \mbox{for any $j\in \mathbb{N}$}.
     \label{eq:kryloverror} 
 \end{equation}
\section{Compound Krylov Subspaces}
\label{sec:pKrylov}
In this section, we define three families of compound Krylov subspaces that are used to solve \eqref{eq:linsys}. We begin by defining the family $\{ CK_{j} \} \subset \mathbb{R}^n$, $CK_{j} = CK_{j}(A,\vec{b})$ that satisfies  
\begin{equation}
\label{eq:exact_cond}
K_j(A(\vec{\sigma}),\vec{b}) \subset CK_{j}(A,\vec{b}) \quad \mbox{for any $\vec{\sigma} \in S$ and $j\in \mathbb{N}$}. 
\end{equation} 
Due to the inclusion in \eqref{eq:exact_cond} and the best approximation property \eqref{eq:proj_error}, the subspace solution to \eqref{eq:linsys} from $CK_j$ is at least as accurate as the $j$th CG-iterate for any $\vec{\sigma} \in S$. 

The subspace satisfying \eqref{eq:exact_cond} that has the smallest possible dimension is 
\begin{equation}
\label{eq:apu_set}
    \bigcup_{k=0}^{j-1} \mathop{span} \{ \; A(\vec{\sigma})^{k} \vec{b} \; | \; \vec{\sigma} \in S \; \}.
\end{equation}
We take advantage of linearity of $A$, and define the $CK$ subspace \emph{containing} \eqref{eq:apu_set} by linearisation: the terms $A(\vec{\sigma})^k \vec{b}$ are written as $A(\vec{\sigma})^k\vec{b} = L_k \vec{\sigma} \otimes \cdots \otimes \vec{\sigma}$, where $L_k$ is the $k$th linearisation matrix and the Kronecker product is repeated $k$-times, see Section~\ref{sec:exact}. By definition, it holds that $\mathop{span} \{ \; A(\vec{\sigma})^k \vec{b} \; | \; \sigma \in S \; \} \subset \mathop{range} (L_k)$. Hence, we  define
\begin{equation}
\label{eq:CK_apu}
    CK_j(A,\vec{b}) := \bigcup_{k=0}^{j-1} \mathop{range} (L_{k}). 
\end{equation}
The column dimension of $L_k$ depends exponentially on $k$ whereas its row dimension is fixed. Thus, we compute $\mathop{range}(L_k)$ using the normal form $L_k L_k^T \in \mathbb{R}^{n\times n}$ that can be formed \emph{without} ever constructing $L_k$, see Lemma~\ref{lemma:normal_form} and Remark~\ref{remark:normalform}. 

The computational cost of a subspace method depends on the dimension of the applied method subspace. To keep both small, we propose two families of \emph{approximate compound Krylov subspaces}, denoted by $\{ CK_{j}^1 \}$ and $\{ CK_{j}^2 \}$, that have smaller dimensions but admit similar error estimate as $\{CK_{j}\}$. Spaces $CK^{1}_{j}$ are obtained by using the span of the most dominant right singular vectors of the linearisation matrix $L_k$ instead of $\mathop{range}(L_k)$ in \eqref{eq:CK_apu}. The space $CK_{j}^2$ is obtained by applying an approximate linearisation process including a low-rank approximation step. In both cases, the low-rank approximation can be implemented in a way that eliminates the exponential growth in the dimension of all involved matrices for a favorable $A$, see Lemma~\ref{lemma:normal_form}, Theorem~\ref{thm:C2L}, and numerical examples in Section~\ref{sec:num}.
\subsection{Linearisation process}
\label{sec:exact}
Next, we discuss the linearisation process and define the family of exact-CK subspaces satisfying \eqref{eq:CK_apu}. We begin with some notation. 

\begin{definition} \label{def:k_kron} Let $k\in \mathbb{N}$ and $\vec{\sigma} \in \mathbb{R}^s$. We write $\vec{\sigma}^{\otimes k}$ for the $k$-times Kronecker product of $\vec{\sigma}$,
\begin{equation*}
\label{eq:k_kron}
    \vec{\sigma}^{\otimes k} := \begin{cases} \vec{\sigma} \otimes \vec{\sigma}^{\otimes (k-1)} & \mbox{for } k > 1 \\ 
    \vec{\sigma} & \mbox{for } k=1 \end{cases}.
\end{equation*}
\end{definition}
\noindent The \emph{linearisation function} $L$ of $A$ is defined as follows: 
\begin{definition} 
\label{def:Lfun} 
Let the matrix-valued function $A:\mathbb{R}^s \mapsto \mathbb{R}^{n \times n}$ be linear and $\{A_i\}_{i=1}^s \subset \mathbb{R}^{n\times n}$ such that $A(\vec{\sigma}) = \sum_{i=1}^s A_i \sigma_i$ for any $\vec{\sigma} \in S$. The linearisation function $L: \mathbb{R}^{n\times m} \mapsto \mathbb{R}^{n\times s m}$ of $A$ is defined as 
 \begin{equation*}
     L(C) = \begin{bmatrix} A_1 C & A_2 C & \ldots & A_s C \end{bmatrix}
 \end{equation*}
 for any $C \in \mathbb{R}^{n\times m}$.
 \end{definition}
\noindent We write $L^k$ for the functional power,  i.e.,  $L^k:=L\circ L^{k-1}$ for $k\in \mathbb{N}, k > 1$ and $L^1 = L$. Using induction, we obtain the following Lemma that gives the linearisation of $A(\vec{\sigma})^{k} \vec{b}$ with respect to $\vec{\sigma}$.
\begin{lemma} 
\label{lemma:lin} Let $A:\mathbb{R}^s \mapsto \mathbb{R}^{n\times n}$ be linear, and $L: \mathbb{R}^{n\times m} \mapsto \mathbb{R}^{n\times s m}$ be the linearisation function of $A$.  Then 
 \begin{equation*}
     A(\vec{\sigma})^k \vec{b} = L^k(\vec{b}) \vec{\sigma}^{\otimes k}  \quad \mbox{for any $k \in \mathbb{N}$ and $\vec{\sigma} \in \mathbb{R}^s$}
     \label{eq:Lrecursion}
 \end{equation*}
\end{lemma}
Observe that the column dimension  of $L^k(\vec{b})$ grows exponentially with $k$, particularly, $L^k(\vec{b}) \in \mathbb{R}^{n \times s^k}$.
\begin{proof} The proof follows by induction. By \eqref{eq:sum_rep} and Definition~\ref{def:Lfun},
\begin{equation*}
     A(\vec{\sigma} ) \vec{b} = \sum_{i=1}^s \sigma_i A_i \vec{b}  = L(\vec{b}) \vec{\sigma}.
     \label{eq:L}
\end{equation*}
Let $k\in \mathbb{N}$, and assume that $A(\vec{\sigma})^k \vec{b} = L^k(\vec{b}) \vec{\sigma} ^{\otimes k}$ holds. Then,
\begin{equation*}
    A(\vec{\sigma})^{k+1} \vec{b} = A(\vec{\sigma}) L^k(\vec{b}) \vec{\sigma} ^{\otimes k} =  \sum_{i=1}^s \sigma_i A_i L^k(\vec{b}) \vec{\sigma} ^{\otimes k},
\end{equation*}
and further
\begin{equation*}
    A(\vec{\sigma})^{k+1} \vec{b} = \begin{bmatrix} A_1 L^k(\vec{b}) & \cdots & A_s L^k(\vec{b}) \end{bmatrix} 
    \begin{bmatrix} \sigma_1 \vec{\sigma} ^{\otimes k} \\ \hdots \\ \sigma_s \vec{\sigma} ^{\otimes k} \end{bmatrix}.
\end{equation*}
Recalling the definition of the Kronecker product and linearisation function of $A$ completes the proof.
\end{proof}
\noindent In practical computation, the subspace $CK_{j+1}$ is obtained by augmenting $CK_{j}$ with $\mathop{range}\left(L^j(\vec{b}\right)$. Hence, we use the following recursive definition instead of \eqref{eq:CK_apu}:
\begin{definition}
\label{process:exact} Let $\vec{b} \in \mathbb{R}^n$, $A:\mathbb{R}^s \mapsto \mathbb{R}^{n \times n}$ be linear, and $L$ be the linearisation function of $A$. Then the family of compound Krylov subspaces  $CK_{j}(A,\vec{b})$ is defined as
\begin{equation*}
CK_{1} = \mathop{span}(\vec{b}) \quad \mbox{and} \quad CK_{j+1} = \mathop{range}\left( L^j(\vec{b}) \right) \oplus CK_{j} \quad \mbox{for $j\in \mathbb{N}$, $j>1$}.
\end{equation*}
\end{definition}
\noindent By Lemma~\ref{lemma:lin}, it holds that $K_j(A(\vec{\sigma}),\vec{b}) \subset CK_{j}$ for any $\vec{\sigma} \in S$ and $j \in \mathbb{N}$. Let $\vec{\sigma} \in S$ and $\hat{\vec{x}}_j(\vec{\sigma})$ be the subspace solution to \eqref{eq:linsys} from $CK_{j}$. By the best approximation property~\eqref{eq:proj_error}, $\hat{\vec{x}}_j(\sigma)$ admits identical error estimate with the CG method, this is, 
\begin{equation*}
    \| \vec{x}(\vec{\sigma}) - \hat{\vec{x}}_j(\vec{\sigma}) \|_{A(\sigma)} \leq  2\left( \frac{\sqrt{\kappa(A(\vec{\sigma}))}-1}{\sqrt{\kappa(A(\vec{\sigma}))}+1}\right)^j \| \vec{x}(\vec{\sigma}) \|_{A(\sigma)}
    \quad \mbox{for any $\vec{\sigma} \in S$ and $j\in \mathbb{N}$}.
\end{equation*}

Computing a basis of $CK_{j}$ requires evaluating $\mathop{range}\left(L^k(\vec{b})\right)$ for $k\in\{1,\ldots,j\}$, e.g. by utilising SVD. When doing so, one has to decide which singular values correspond to zero and which do not. For this reason, the authors advice one to use the family of spaces $\{CK_{j}^1\}$, whose definition includes such SVD truncation step, instead of $\{CK_{j}\}$, see Section~\ref{sec:direct_svd}.

Recall that $L^k(\vec{b}) \in \mathbb{R}^{n \times s^k}$ for $k\in \mathbb{N}$. Due to the exponentially growing column dimension, the matrices $L^k(\vec{b})$ are analytical tools that should be avoided in any practical implementation. We explain in Section~\ref{sec:direct_svd} how $\mathop{range}\left(L^k(\vec{b})\right)$ can be computed from the normal form $L^k(\vec{b}) L^k(\vec{b})^T \in \mathbb{R}^{n\times n}$  without constructing $L^k(\vec{b})$.  

\subsection{Direct approximation}
\label{sec:direct_svd}

Next, we define the family of approximate compound Krylov subspaces of the first kind $\{CK_j^1\}$, and give an error estimate for the subspace solution to \eqref{eq:linsys} from $CK_j^1$. We begin with some notation.
\begin{definition} \label{def:low_rank} Let $B \in \mathbb{R}^{n\times m}$, $\delta >0$, and $B=U\Sigma V^T$ be the SVD of $B$. Assume that the singular values of $B$ are in non-increasing order, let $r \in \mathbb{N}$ satisfy 
\begin{equation*}
    \sigma_{r+1} < \delta \leq \sigma_{r}, 
\end{equation*}
$U_{r} = U(:,1:r)$, $\Sigma_r = \Sigma_r(1:r,1:r)$, and $V_r = V(1:r,:)$. We call the matrix $\widehat{B} = U_r \Sigma_r V_r^T$ as well as the triplet $(U_r,\Sigma_r,V_r)$ as the $\delta$-accurate low-rank approximation of $B$. The matrix $\widehat{B}$ satisfies the error estimate
\begin{equation}
\label{eq:low_error}
\| B - \widehat{B} \|_2 = \sigma_{r+1} < \delta.
\end{equation}
\end{definition}
The approximation property in \eqref{eq:low_error} is stated in the $2$-norm. We obtain approximation result in the $\| \cdot \|_{A(\vec{\sigma})}$ - norm by using the norm equivalence 
\begin{equation}
    \label{eq:nequiv}
\lambda_{min}(A(\vec{\sigma})) \| \vec{z} \|_2 \leq \| \vec{z} \|_{A(\vec{\sigma})} \leq \lambda_{max}(A(\vec{\sigma})) \| \vec{z} \|_2 
\end{equation}
valid for any $\vec{z} \in \mathbb{R}^n$ and $\vec{\sigma} \in S$. Let $\widehat{B} \in \mathbb{R}^{n\times n}$ be the $\delta$-accurate low-rank approximation of $B$. By \eqref{eq:nequiv} and the definition of the operator-norm it holds that 
\begin{equation}
    \| B-\widehat{B} \|_{A(\vec{\sigma})} \leq \delta \kappa_2(B) .
\end{equation}

For notational convenience, denote $L_j := L^j(\vec{b})$ for $j \in \mathbb{N}$ and set $L_0 := \vec{b}$. We arrive to the definition of approximate parametric Krylov subspace of the first kind. 
\begin{definition} 
\label{process:svd}  Let $j \in \mathbb{N}$, $\vec{b} \in \mathbb{R}^n$, $A:\mathbb{R}^s \mapsto \mathbb{R}^{n \times n}$ be linear and $L$ be the linearisation function of $A$. In addition, let $\{\delta_k\}_{k=1}^{j-1}$ be a set of positive cut-off tolerances, $L_{0r}:=\vec{b}$, and $L_{kr}$ the $\delta_k$-accurate low-rank approximation of $L_k$ for any~$k\in \{1,\ldots,j-1\}$, see Definition~\ref{def:low_rank}. The space $CK^{1}_{j} = CK^{1}_{j}(A,\vec{b},\{\delta_k\}_{k=1}^{j-1})$ is defined as 
\begin{equation*}
CK^1_{1} = \mathop{span}(\vec{b}) \quad \mbox{and} \quad CK^1_{k+1} = \mathop{range}(L_{kr}) \oplus CK^1_{k} \quad \mbox{for $k \in \{1,\ldots,j-1\}$}.
\end{equation*}
\end{definition}
\begin{remark} Our analysis states that improving the accuracy of the subspace solution from $CK^{1}_j(A,b)$ requires changing $\{ \delta_k \}_{k=1}^{j-1}$ as well as increasing $j$. For this reason, the definition is given for the space $CK_{j}^1$, not for the whole family of spaces.
\end{remark}
We proceed to derive error estimate for the subspace solution to \eqref{eq:linsys} from $CK_{j}^1$. First, recall that the error analysis for the $j$th CG-iterate in Section~\ref{sec:cg} is based on constructing a solution candidate so that the difference between the exact solution and the candidate can be analytically estimated. Repeating steps given in Section~\ref{sec:cg} for problem~\eqref{eq:linsys} with $\vec{\sigma} \in S$ yields the solution candidate $\vec{v}^*_j(\vec{\sigma}) \in K_j(A(\vec{\sigma}),\vec{b})$ defined similar to~\eqref{eq:vstar}, 
\begin{equation*}
    \vec{v}_j^*(\vec{\sigma}) = \sum_{k=1}^j \gamma_{jk}(\vec{\sigma}) A(\vec{\sigma})^{k-1} \vec{b}, 
\end{equation*}
The coefficients $\gamma_{jk}(\vec{\sigma})$ depend on the index $j$ and on the largest as well as the smallest eigenvalue of the matrix $A(\vec{\sigma})$ as in \eqref{eq:vstar}. 

Let $\hat{\vec{x}}_j(\vec{\sigma})$ be the subspace solution to \eqref{eq:linsys} from $CK_{j}^1$. By the best approximation property \eqref{eq:proj_error} and triangle inequality, 
\begin{equation}
\label{eq:apu0}
    \| \hat{\vec{x}}_j(\vec{\sigma}) -\vec{x}(\vec{\sigma}) \|_{A(\vec{\sigma})} \leq \| \vec{v}_j^*(\vec{\sigma}) -\vec{x}(\vec{\sigma}) \|_{A(\vec{\sigma})} + \min_{\vec{v} \in CK_{j}^1} \| \vec{v}_j^*(\vec{\sigma}) - \vec{v} \|_{A(\vec{\sigma})}
\end{equation}
for any $\vec{\sigma} \in S$. The first term on the RHS of \eqref{eq:apu0} is estimated as a part of CG error analysis given in Section~\ref{sec:cg}. The latter term measures how accurately the solution candidate $\vec{v}_j^*(\vec{\sigma})$ can be approximated in $CK_j^1$. The size of this term depends on the choice of the cut-off tolerances $\{\delta_k\}_{k=1}^{j-1}$, and it is estimated in the following Lemma.
\begin{lemma} \label{lemma:v*CK1} Make the same assumptions and use the same notation as in Definition~\ref{process:svd}. Let $\vec{\sigma} \in S$ and $\vec{v}^*_j(\vec{\sigma})$ be as defined in \eqref{eq:vstar}. Then there holds that
\begin{equation}
\label{eq:svdlemma1}
\min_{\vec{v} \in CK_{j}^1} \| \vec{v}_j^*(\vec{\sigma}) - \vec{v} \|_{A(\vec{\sigma})} \leq \lambda_{max}(A(\vec{\sigma})) \sum_{k=1}^{j-1} \delta_k |\gamma_{j(k+1)}(\vec{\sigma})| \| \vec{\sigma}^{\otimes k} \|_2
\end{equation}
for any $j\in \mathbb{N}$ and $\vec{\sigma} \in S$. The coefficients $\gamma_{jk}(\vec{\sigma}) \equiv \gamma_{jk}(\lambda_\mathrm{min}(A(\vec{\sigma})),\lambda_\mathrm{max}(A(\vec{\sigma})))$ are as defined in \eqref{eq:cg_cheb}.
\end{lemma}
\begin{proof}  By definition \eqref{eq:vstar}, 
\begin{equation*}
\vec{v}_j^*(\vec{\sigma})
= \sum^j_{k=1} \gamma_{jk}(\vec{\sigma}) A^{k-1}(\vec{\sigma}) \vec{b}
= \gamma_{j1} \vec{b} + \sum^{j-1}_{k=1} \gamma_{j(k+1)}(\vec{\sigma}) L_{k} \vec{\sigma}^{\otimes k}.
\end{equation*}
By Definition~\ref{process:svd}, $L_{kr} \vec{\sigma}^{\otimes k} \in CK_{j}^1$ for any $k \in \{1,\ldots,j-1 \}$. The estimate \eqref{eq:svdlemma1} follows by choosing $\vec{v}$ in \eqref{eq:apu} as $\vec{v} = \gamma_{j1} \vec{b} + \sum_{k=1}^{j-1} \gamma_{j(k+1)} L_{kr} \vec{\sigma}^{\otimes k}$ so that 
\begin{equation*}
\vec{v}_j^*(\vec{\sigma}) - \vec{v}
= \sum^{j-1}_{k=1} \gamma_{j(k+1)} (L_k - L_{kr}) \vec{\sigma}^{\otimes k},
\end{equation*}
and then utilising~\eqref{eq:nequiv}, triangle inequality, and \eqref{eq:low_error}.
\end{proof}
\noindent The above discussion is summarised in the following Theorem:
\begin{theorem}
\label{thm:error1}
Let $j\in \mathbb{N}$, $\vec{b} \in \mathbb{R}^n$, $A:\mathbb{R}^s \mapsto \mathbb{R}^{n \times n}$ be linear, and $A(S) \subset \mathbb{S}^n_{++}$. In addition, let $tol > 0$, the cut-off tolerances $\{\delta_k\}_{k=1}^{j-1}$ satisfy 
\begin{equation}
\label{eq:error1_deltak}
   \lambda_{max}(A(\vec{\sigma})) \sum_{k=1}^{j-1} \delta_k |\gamma_{j(k+1)}(\vec{\sigma})| \| \vec{\sigma}^{\otimes k} \|_2 \leq tol
\end{equation}
and $CK_{j}^1$ be the corresponding approximate compound Krylov subspace of the first  kind. Then the subspace solution $\hat{\vec{x}}_j(\vec{\sigma})$ to \eqref{eq:linsys} from $CK_{j}^1$ satisfies
\begin{equation*}
 \| \hat{\vec{x}}_j(\vec{\sigma}) -\vec{x}(\vec{\sigma}) \|_{A(\vec{\sigma})}  \leq 2\Bigg{(}\frac{\sqrt{\kappa(A(\vec{\sigma)})}-1}{\sqrt{\kappa(A(\vec{\sigma}))}+1}\Bigg{)}^j \| \vec{x}(\vec{\sigma}) \|_{A(\vec{\sigma})} + tol.
\end{equation*}
for any $\vec{\sigma} \in S$. 
\end{theorem}
According to Theorem~\ref{thm:error1}, cut-off tolerances that yield error level $tol$ depend on $\| \vec{\sigma}^{\otimes k} \|_2$ and the coefficients $\{\gamma_{j(k+1)}\}_{k=1}^{j-1}$, that can both be be large. Particularly $\| \vec{\sigma}^{\otimes k} \|_2 = \| \vec{\sigma} \|^k_2$, which is large if $\|\vec{\sigma}\|_2$ is. In Section~\ref{sec:fem}, each component of $\vec{\sigma}$ is allowed to vary between given bounds, hence, it is natural to estimate
\begin{equation}
    \| \vec{\sigma} \|_2 \leq \sqrt{s} \| \vec{\sigma} \|_{\infty} \quad \mbox{so that} \quad 
    \| \vec{\sigma}^{\otimes k} \|_2 \leq s^{k/2} \| \vec{\sigma} \|^k_{\infty}.
\end{equation}
We do not estimate the size of the coefficients $|\gamma_{jk}(\vec{\sigma})|$, but regard them as constants that depend on $S$, $j$, and $A$. 
%

As the column dimension of $L_k$ increases exponentially with $k$, these matrices cannot be constructed in practice. Next, we discuss a feasible computational strategy for constructing a basis for $CK^1_{j}$ that does not use $L_k$. 

By Definition~\ref{process:svd}, a basis for $CK_{k}^1$ is obtained from the bases of $CK_{k-1}^1$ and $\mathop{range}(L_{kr})$. Recall that 
 \begin{equation}
 \label{eq:LLT}
     L_k L_k^T = U_k \Lambda_k U_k^T \quad \mbox{for} \quad  \Lambda = \Sigma_k \Sigma_k^T = \mathop{diag}(\sigma_1^2,\ldots,\sigma_n^2),
 \end{equation}
where $U_k$ is the matrix of right singular vectors, and $\{ \sigma_i \}_{i=1}^n$ is the set of singular values of $L_k$. Hence, the range of $L_{kr}$ is obtained by computing largest eigenvalues and eigenvectors of \emph{the normal form} $L_k L_k^T$. The next Lemma gives a way to compute $L_k L_k^T$ without explicitly constructing $L_k$.
 
\begin{lemma} \label{lemma:normal_form} Let $L:\mathbb{R}^{n\times m} \mapsto \mathbb{R}^{n\times sm}$ be as in Definition~\ref{def:Lfun} and $\vec{b} \in \mathbb{R}^n$. Denote $L_k = L^k(\vec{b})$ and let $L_0L_0^T = \vec{b} \vec{b}^T$. Then there holds that
\begin{equation}
L_k L_k^T = \sum_{i=1}^s A_i L_{k-1}L_{k-1}^T A_i^T \quad \mbox{for any $k \in \mathbb{N}$}.
\end{equation}
\end{lemma}
\begin{proof} By induction. \end{proof}  
\begin{remark}
\label{remark:normalform} Algorithms based on normal forms are avoided in numerical linear algebra due to numerical stability issues arising from finite precision arithmetics and increased condition number, see \cite{Higham:2002}. The methods proposed in this work are intended to be used in connection with the finite element method. As the error due to finite element discretization is typically much larger than error due to finite precision, we are not concerned with numerical stability. 
\end{remark}

\begin{remark}
\label{remark:projection}
In our preliminary numerical experiments we observed that the intersection of $\mathop{range}(L_{kr})$ and $CK_{k-1}^1$ can have a large dimension. To reduce the cost of computing the required eigenpairs of $L_k L_k^T$, we eliminate all information that is already contained in $CK_{k-1}^{1}$ by orthogonal projection in the $2$-inner product $\Pi_k$ to $CK_{k-1}^{1}$. This is, we construct a low-rank approximation to range of the operator $(I-\Pi_k) L_k$. Such strategy leads to identical error estimate as the one given in Theorem~\ref{thm:error1}, but is possibly more efficient. As $\mathop{range} (I-\Pi_k) L_k$ is orthogonal to $CK_{k-1}^{1}$, the projection also simplifies the construction of $CK_{k}^1$.  
\end{remark}

\subsection{Intermediate approximation}
\label{sec:inter_svd}
In this section, we define the family of approximate compound Krylov subspaces of the second kind $\{ CK_{j}^2\}$. Similar method is used in \cite{BaCo:17} to define Neumann series based approximation of the parameter-to-solution map for parametric operator equation. First, recall the linearisation given in Lemma~\ref{lemma:lin}:
\begin{equation*}
    A(\vec{\sigma})^{k} \vec{b} = L^k(\vec{b}) \vec{\sigma}^{\otimes k} \quad \mbox{for $k\in \mathbb{N}$ and $\vec{\sigma} \in \mathbb{R}^s$},
\end{equation*}
where $L$ is the linearisation function of $A$ and $L^k(\cdot)$ denotes the $k$th functional power evaluated recursively as 
\begin{equation}
\label{eq:tmp1}
L^k(\cdot) = L(L^{k-1}(\cdot)) \quad\mbox{for $k\in \mathbb{N}$, $k>1$}.
\end{equation}
The approximate CK subspaces of the first kind are obtained by a low-rank approximation of $L^k(\vec{b})$, whereas subspaces of the second kind as constructed by including a low-rank approximation step to the recursion in~\eqref{eq:tmp1}. We proceed by defining a sequence of \emph{approximate linearisation matrices} $\{\widehat{L}_k\}$. 
 \begin{definition} \label{process:CK2} 
 Let $j\in \mathbb{N}$, $\vec{b} \in \mathbb{R}^n$, $A:\mathbb{R}^s \mapsto \mathbb{R}^{n\times n}$ be linear, and $L$ be the linearisation function of $A$. In addition, let $\{\delta_k\}_{k=1}^{j}$ be the set of cut-off tolerances. The sequence of approximate linearisation matrices $\{\widehat{L}_k\}_{k=1}^j$ of $(A$, $\vec{b}$, $\{\delta_k\}_{k=1}^j)$  is  defined as follows: $\widehat{L}_0 := \vec{b}$, and $\widehat{L}_{k}$ is the $\delta_k$-accurate low-rank approximation of $L(\widehat{L}_{k-1})$ for $k \in \{1,\ldots,j\}$. 
 \end{definition} 
\begin{definition} \label{process:interspace} Make the same assumptions and use the same notation as in Process \ref{process:CK2}. Particularly, let $j\in \mathbb{N}$ and $\{\widehat{L}_k\}_{k=1}^{j-1}$ be the sequence of approximate linearisation matrices of $(A$, $\vec{b}$, $\{\delta_k\}_{k=1}^{j-1})$. Then 
\begin{equation}
CK_{1}^{2} = \mathop{span}(\vec{b}) \quad \mbox{and} \quad CK_{k+1}^{2} = \mathop{range}(\widehat{L}_{k}) \oplus CK^2_{k} \quad \mbox{for $k\in \{1,\ldots,j-1\}$.}
\label{eq:intermedspace}
\end{equation}
\end{definition}

Let $\vec{\sigma} \in S$ and $\hat{\vec{x}}_j(\vec{\sigma})$ be the subspace solution from $CK_{j}^2$. Next, we estimate the error $\vec{x}(\vec{\sigma}) - \hat{\vec{x}}_j(\vec{\sigma})$ by using the same approach as in Section~\ref{sec:direct_svd}. This is, we study how accurately $\vec{v}^*_j(\vec{\sigma})$, defined in \eqref{eq:vstar}, can be approximated in $CK^2_{j}$. First, the error related to the approximation $A(\vec{\sigma})^k \vec{b} \approx \widehat{L}_{k} \vec{\sigma}^{\otimes j}$ is bounded by using the properties of the linearisation function:
\begin{lemma} 
\label{lemma:Lprop} Let $A:\mathbb{R}^s \mapsto \mathbb{R}^{n\times n}$ be linear and $L$ be the linearisation function of $A$. Then for any $\vec{\sigma} \in S$, $k \in \mathbb{N}$ and $B_1,B_2 \in \mathbb{R}^{n\times m}$ it holds that:

\renewcommand{\labelenumi}{(\roman{enumi})}
\begin{enumerate}
    \item $A(\vec{\sigma}) B_1 \vec{\sigma}^{\otimes k} = L( B_1 ) \vec{\sigma}^{\otimes (k+1) }$
    \end{enumerate}
    \noindent and 
    \begin{enumerate}
    \setcounter{enumi}{1}
    \item $L(B_1 + B_2) = L(B_1) + L(B_2)$.
\end{enumerate}
\end{lemma}
\begin{lemma} \label{lemma:approx_lin} Let $j\in \mathbb{N}$, $\vec{b} \in \mathbb{R}^n$, $A:\mathbb{R}^s \mapsto \mathbb{R}^{n\times n}$ be linear, and $L$ the linearisation function of $A$. In addition, let $\{\delta_k\}_{k=1}^j$ be the set of cut-off tolerances and $\{ \widehat{L}_k \}_{k=1}^j$ the approximate linearisation matrices of $(A$,$\vec{b}$,$\{\delta_k\}_{k=1}^j)$ as in Definition~\ref{process:CK2}. Then 
\begin{equation*}
\|\widehat{L}_{j} \vec{\sigma}^{\otimes j} - A(\vec{\sigma})^j \vec{b} \|_2 \leq \sum_{k=1}^j \delta_k \| A(\vec{\sigma}) \|_2^{j-k}  \| \vec{\sigma}^{\otimes k} \|_2.
\end{equation*}
for any $j\in \mathbb{N}$.
\end{lemma}
\begin{proof} Let $\{L_k\}$ be the linearisation matrices of $(A,\vec{b})$ so that $A(\vec{\sigma})^k \vec{b} = L_k \vec{\sigma}^{\otimes k}$. There holds that 
\begin{equation*}
\left( \widehat{L}_{k+1} - L_{k+1} \right) \vec{\sigma}^{\otimes (k+1)} = 
L(\widehat{L}_{k} - L_{k}) \vec{\sigma}^{\otimes (k+1)} + \left( \widehat{L}_{k+1} - L(\widehat{L}_k) \right) \vec{\sigma}^{\otimes( k+1)}.
\end{equation*}
Using properties of the linearisation function in Lemma~\ref{lemma:Lprop} yields
\begin{equation*}
(\widehat{L}_{k+1} - L_{k+1}) \vec{\sigma}^{\otimes(k+1)} = A(\vec{\sigma} ) (\widehat{L}_{k}-L_{k}) \vec{\sigma}^{\otimes k} + \left( \widehat{L}_{k+1} - L(\widehat{L}_k) \right) \vec{\sigma}^{\otimes (k+1)}.
\end{equation*}
Denote $\xi_k := \|  (\widehat{L}_{k} - L_{k}) \vec{\sigma}^{\otimes k} \|$. Then  
\begin{equation*}
\xi_{k+1} \leq  
\| A(\vec{\sigma}) \| \eta_k + \delta_{k+1} \| \vec{\sigma}^{\otimes k+1} \|. 
\end{equation*}
and $\xi_0 = 0$. Solving this non-homogeneous recurrence relation completes the proof.
\end{proof}
Combining the approximation result in Lemma~\ref{lemma:approx_lin} with technique used in the proof of Lemma~\ref{lemma:v*CK1} gives an error estimate for the subspace solution from the compound Krylov subspace of the second kind.
\begin{theorem}
\label{thm:error2} 
Let $j\in \mathbb{N}$, $\vec{b} \in \mathbb{R}^n$, $A:\mathbb{R}^s \mapsto \mathbb{R}^{n \times n}$ be linear, and $A(S) \subset \mathbb{S}^n_{++}$. In addition, let $tol > 0$, the cut-off tolerances $\{\delta_k\}_{k=1}^{j-1}$ satisfy 
\begin{equation}
\label{eq:error2_deltak}
   \lambda_{max}(A(\vec{\sigma})) \sum_{k=1}^{j-1} \left( |\gamma_{j(k+1)}(\vec{\sigma})|  \sum_{l=1}^k \delta_l \| A(\vec{\sigma}) \|_2^{k-l}  \| \vec{\sigma}^{\otimes l} \|_2 \right) \leq  tol,
\end{equation}
for any $\vec{\sigma} \in S$, and $CK_{j}^2$ be the corresponding approximate compound Krylov subspace of the second kind. Then the subspace solution $\hat{\vec{x}}_j(\vec{\sigma})$ to \eqref{eq:linsys} from $CK_{j}^2$ satisfies
\begin{equation*}
 \| \hat{\vec{x}}_j(\vec{\sigma}) -\vec{x}(\vec{\sigma}) \|_{A(\vec{\sigma})}  \leq 2\Bigg{(}\frac{\sqrt{\kappa(A(\vec{\sigma)})}-1}{\sqrt{\kappa(A(\vec{\sigma}))}+1}\Bigg{)}^j \| \vec{x}(\vec{\sigma}) \|_{A(\vec{\sigma})} + tol.
\end{equation*}
for any $\vec{\sigma} \in S$. 
\end{theorem}
\begin{proof} %
The proof follows by the best approximation property~\eqref{eq:chebest} and estimating how accurately $\vec{v}^*_j$ can be approximated in $CK^2_j$. By definition \eqref{eq:vstar}, 
\begin{equation*}
\vec{v}_j^*(\vec{\sigma})
= \sum^j_{k=1} \gamma_{jk}(\vec{\sigma}) A^{k-1}(\vec{\sigma}) \vec{b}
= \gamma_{j1} \vec{b} + \sum^{j-1}_{k=1} \gamma_{j(k+1)}(\vec{\sigma}) 
A^{k}(\vec{\sigma}) \vec{b}.
\end{equation*}
Let $\vec{v} = \gamma_{j1} \vec{b} + \sum_{k=1}^{j-1} \gamma_{j(k+1)}(\vec{\sigma}) \widehat{L}_{k} \vec{\sigma}^{\otimes k}$. Then
\begin{equation*}
\vec{v}_j^*(\vec{\sigma}) - \vec{v}
= \sum^{j-1}_{k=1} \gamma_{j(k+1)}(\vec{\sigma}) ( A(\vec{\sigma})^k - \widehat{L}_{k} \vec{\sigma}^{\otimes k} ).
\end{equation*}
Lemma~\ref{lemma:approx_lin} states that 
\begin{equation}
\label{eq:apu}
\|\widehat{L}_{k} \vec{\sigma}^{\otimes k} - A(\vec{\sigma})^k \vec{b} \|_2 \leq \sum_{l=1}^k \delta_l \| A(\vec{\sigma}) \|_2^{j-l}  \| \vec{\sigma}^{\otimes l} \|_2.
\end{equation}
Error estimate follows by utilising~\eqref{eq:nequiv}, triangle inequality, and \eqref{eq:apu}
\end{proof}
Observe that $\widehat{L}_k \in \mathbb{R}^{n \times s^k}$, i.e., it's  dimension increases exponentially with $k$. Next, we give a practical method for computing a basis for $CK_j^2$ without using matrices $\widehat{L}_k$. Instead, we use another sequence $\{\widehat{C}_k\}$ satisfying
\begin{equation}
\label{eq:rangeLC}
\mathop{range}(\hat{C}_k) = \mathop{range}(\hat{L}_k) \quad \mbox{for all $k\in \mathbb{N}$}.
\end{equation}
 \begin{definition} 
 \label{process:inter2} 
 Let $A:\mathbb{R}^s \mapsto \mathbb{R}^{n\times n}$ be linear and $L$ be the linearisation function of $A$. In addition, let $\{\delta_k\}_{k=1}^{j}$ be the set of cut-off tolerances. The sequence $\{\widehat{C}_k\}_{k=1}^{j}$ associated to triplet $(A, \vec{b}, \{\delta_k\}_{k=1}^{j})$ is defined as follows: $\widehat{C}_0 = \vec{b}$, and $\widehat{C}_{k+1} = \widehat{U}_{kr} \widehat{\Sigma}_{kr}$, where $(\widehat{U}_{kr} ,\widehat{\Sigma}_{kr},\widehat{V}_{kr})$ is the $\delta_k$-accurate low-rank approximation of $L(\widehat{C}_k)$. 
\end{definition} 
 \begin{theorem}
 \label{thm:C2L} Let $\{ \widehat{L}_k \}_{k=0}^j$ and $\{\widehat{C}_j\}_{k=0}^{j}$ be as in Definitions \ref{process:CK2} and \ref{process:inter2}, respectively. Then for any $k \in \{0,\ldots, j\}$ there exists a unitary $Q_k$ such that 
  \begin{equation}
    \label{eq:dropU}
    \widehat{C}_k = \widehat{L}_k Q_k^T. 
\end{equation} 
 \end{theorem}
\noindent In other words, the condition \eqref{eq:rangeLC} holds.
 \begin{proof} The proof is by induction. By definition $\widehat{C}_0 = \widehat{L}_0 = \vec{b}$. Thus \eqref{eq:dropU} holds with $U_0$ being the identity matrix. Next, assume that \eqref{eq:dropU} holds for some $k\in \mathbb{N}$, i.e., $\widehat{C}_k = \widehat{L}_{k} Q_k$ for some unitary $Q_k$. By Processes~\ref{process:CK2} and \ref{process:inter2}, $\widehat{L}_{k+1}$ and $\widehat{C}_{k+1}$ are obtained as $\delta_{k+1}$-accurate low rank approximations of $L(\widehat{L}_{k})$ and $L(\widehat{C}_k)$, respectively. By relation \eqref{eq:dropU}, 
 \begin{equation*}
     L(\widehat{C}_k) = L( \widehat{L}_k Q_k) = \begin{bmatrix}A_1 \widehat{L}_k Q_k^T & \cdots & A_s \widehat{L}_k Q_k^T \end{bmatrix} = L(\widehat{L}_k) (I\otimes Q_k^T).
 \end{equation*}
 Let $U_k \Sigma_k V^T_k$ for $\Sigma_k = \mathop{diag}(\sigma_1,\ldots,\sigma_n)$ be the SVD of $L(\widehat{L}_k)$. As $(I \otimes Q_k^T)$ is unitary, $U_k \Sigma_k V_k^T (I \otimes Q_k^T)$ is the SVD of $L(\widehat{C}_k)$. Let $r$ be the cut-off index satisfying $\sigma_{r+1} \leq \delta_{k+1} < \sigma_{r}$. Then
 \begin{equation*}
     \widehat{L}_{k+1} = U_k(:,1:r) \Sigma_k(1:r,1:r) V_k(:,1:r)^T 
 \end{equation*}
 and
 \begin{equation*}
     \widehat{C}_{k+1} = U_k(:,1:r) \Sigma_k(1:r,1:r). 
 \end{equation*}
This is, $Q_{k+1} = V_k(:,1:r)$. 
\end{proof}

\begin{remark} 
\label{remark:projection2} Constructing $CK_j^2$ requires combining bases of two subspaces. To simplify this step, we propose to compute two low-rank approximations:  First low-rank approximation corresponds to $L(C_k)$ and it is used to define $C_{k+1}$. Second low-rank approximation corresponds to $(I-\pi_k)L(C_k)$, where $\pi_k$ is the orthogonal projection in Euclidean inner product to $CK_{k}^2$. The largest singular vectors of  $(I-\pi_k) L(C_k)$ are then used to define $CK_{k+1}^2$. This process admits similar error estimate to the simpler variant analysed in this section. In Section~\ref{sec:practical} we further simplify the construction of $CK^2_k$ by using the low-rank approximation of $(I-\pi_k) L(C_k)$ instead of $C_k$ to compute an approximation to $C_{k+1}$. This variant performs well and is simple, but requires additional error analysis. 
\end{remark}

\section{Computational considerations} 
\label{sec:practical} In this Section, we briefly describe our implementation of the two approximate CK-solvers for \eqref{eq:linsys}. The inputs are $\{A_i\}_{i=1}^s \subset \mathbb{R}^{n\times n}$, $\vec{b} \in \mathbb{R}^n$, cut-off tolerances $\{\delta_k\}_{k=1}^{j-1}$, and the order of the CK space $j$. The tolerances $\{\delta_k\}_{k=1}^{j-1}$ corresponding to the desired accuracy $tol$ can be chosen using Theorems~\ref{thm:error1} and \ref{thm:error2}. However, this is relatively complicated and requires estimates for $\| \vec{\sigma}^{\otimes k}\|_2$, $\kappa(A(\vec{\sigma}))$, coefficients $\gamma_{jk}(\vec{\sigma})$, and $\|A(\vec{\sigma})\|_2$. Hence, in our numerical examples, we simply set $\delta_k = \delta$. The desired index of the Krylov subspace is obtained from upper bound for the condition number $\kappa(A(\vec{\sigma}))$.

The algorithm computing a basis for the compound Krylov subspaces of the first kind described in Sec. \ref{sec:direct_svd} is given in Alg. \ref{alg:CK1}. It uses the orthogonal projection technique of Remark \ref{remark:projection}.  The algorithm for computing a basis for the compound Krylov subspaces of the second kind described in Sec. \ref{sec:inter_svd} is given in Alg. \ref{alg:CK2}. The approximate variant of Remark~\ref{remark:projection2} is used. These algorithms return the basis of $CK^1_j$ or $CK^2_j$ as the columns of the matrix $Q\in \mathbb{R}^{n\times k}$.

Either of the computed basis $Q$ can then be used to compute the solution to the linear equation \eqref{eq:linsys} for any $\vec{\sigma} \in S$. The algorithm for obtaining the solution for a given $\vec{\sigma}$ from a subspace with basis $Q$ is written in Alg. \ref{alg:solve}. This function will be called multiple times with multiple parameter vectors $\vec{\sigma}$.
\begin{algorithm}[H]
   \caption{}
   \label{alg:itemset}
    \begin{algorithmic}[1]
    \Function{CK1}{$\{ A_i \}_{i=1}^s,\vec{b},j,\{\delta_k\}_{k=1}^{j-1}$}
    \State{$Q = \vec{b} \| \vec{b} \|_2^{-1}$} 
    \State{$L_0 L_0^T = \vec{b} \vec{b}^T$}. 
    \For{$k=1,\ldots,j-1$} 
            \State{Compute $L_{k}L_{k}^T$ using Lemma~\ref{lemma:normal_form}}.
            \State{Compute eigendecomposition $U_k \Lambda_k U^T_k = (I-Q Q^T)L_{k} L_{k}^T (I-Q Q^T)$}
            \State{Find index $r$ s.t. $\sigma_{r} > \delta_k >= \sigma_{r+1}$}
            \State{Update $Q = \begin{bmatrix} Q & U_k(:,1:r)\end{bmatrix}$}
    \EndFor
    \State \Return $Q$
    \EndFunction
    \end{algorithmic}
    \label{alg:CK1}
\end{algorithm}

\begin{algorithm}[H]
   \caption{}
   \label{alg:itemset2}
    \begin{algorithmic}[1]
    \Function{CK2}{$\{ A_i \}_{i=1}^s,\vec{b},j,\{\delta_k\}_{k=1}^{j-1}$}
    \State{$Q = \vec{b} \| \vec{b} \|_2^{-1}$} 
    \State{$C_0 = \vec{b}$}. 
    \For{$k=1,\ldots,j-1$} 
            \State{Compute SVD $U_k \Sigma_k V^T_k = (I-QQ^T)L(C_{k-1})$}
            \State{Find index $r$ s.t. $\sigma_{r} > \delta_k >= \sigma_{r+1}$}
            \State{Update $Q = \begin{bmatrix} Q & U_k(:,1:r)\end{bmatrix}$}
            \State{Update $C_k = U_k(:,1:r) \Sigma_k(1:r,1:r)$}
    \EndFor
    \State \Return $Q$
    \EndFunction
    \end{algorithmic}
    \label{alg:CK2}
\end{algorithm}

\begin{algorithm}[H]
   \caption{}
   \label{alg:itemset3}
    \begin{algorithmic}[1]
    \State{Precompute $\widehat{\vec{b}} = Q^T \vec{b}$, \; $\widehat{A}_i = Q^T A_i Q$} 

    \Function{solve}{$\{ \widehat{A}_i \}_{i=1}^s$, $\widehat{\vec{b}}$, $Q$, $\vec{\sigma}$}
    \For{$k=1,\ldots,n$} 
        \State{Compute $\widehat{A} = \sum_{i=1}^s \widehat{A}_i \sigma_i$}
        \State{Solve $\widehat{A} \vec{x} = \widehat{b}$}.
    \EndFor

    \State \Return $\{ Q \vec{x} \}$.
    \EndFunction
    \end{algorithmic}
    \label{alg:solve}
\end{algorithm}

\section{Numerical Examples}
\label{sec:num}

\subsection{Piecewise constant material parameter} First, we demonstrate the proposed CK methods by solving the parametric linear system given in Section \ref{sec:matparam}. We use $N \times M$ checkerboard patterns for multiple values of $N$ and $M$. The parameter set is chosen as $S = \{ \vec{\sigma} \in \mathbb{R}^s \; | \; \sigma_i \in [1,a] \quad \mbox{for $i \in \{1,\ldots,s\}$} \; \}$ for $a=20$. We investigated empirically the error between the exact and the CK subspace solutions for different values of the cut-off tolerance and index $j$. 

To compute the error we used random sampling strategy, where we first constructed a random set of $\vec{\sigma}$ vectors so that each element $\sigma_i$ was drawn from a uniform distribution $\sigma_i \sim \mathcal{U}(1,20)$. We then calculated the error of the compound Krylov method compared to solving the linear system using the Matlab backslash. All errors were calculated in the $A(\vec{\sigma})$-norm. Unless stated otherwise, the dimension of the FE space in this test was $961$. The vector $\vec{b}$ corresponds to the constant loading $f=1$.

We found that using a constant cut-off tolerance $\delta_k$ was sufficient even though our estimates suggest fine-tuning it for each round of iteration separately. Our examples below use the value $\delta_k = \delta = 10^{-7}$.

The largest errors out of 100 randomly drawn vectors $\vec{\sigma}$ for different orders of the CK subspace are plotted in Fig. \ref{fig:errors_checkerboard}. The improvement begins exponentially and then slows down when the error approaches the cut-off tolerance $\delta$. The number of subdomains does not affect the error very much. Observe, that the $4$-subdomain case converges a bit slower but still reaches similar error levels for $j=5$. 

The errors for different randomly drawn $\vec{\sigma}$ vectors are visualised as scatter plots in Fig. \ref{fig:scatter}. The larger value of the variable $\sigma_{\max}/\sigma_{\min}$ is related to a larger condition number $\kappa_2(A(\vec{\sigma}))$, and on average a larger error. For the $2 \times 2$ checkerboard configuration there sometimes are extremely accurate CK-solutions even when the ratio $\sigma_{\max}/\sigma_{\min}$ is large.

The dimensions of the compound Krylov subspaces is given in Table \ref{table:dimensions}. It seems that a larger FE space does not affect the size of the subspace very much. It's worth noting that we've managed to lower the size of the subspace quite significantly from the original FE space dimension. The errors obtained with the cut-off tolerances we used in this example are quite small, in practical applications one might afford to have larger errors and therefore possibly even smaller subspaces.

\begin{table}[h]
    \begin{center}
    \begin{tabular}{ |c|c|c|c| }
    \hline
    \begin{tabular}{@{}c@{}}Number of \\ subdomains\end{tabular} & 
    Method & 
    \begin{tabular}{@{}c@{}}Subspace dimension \\ (FE space dim. 961)\end{tabular}
    & \begin{tabular}{@{}c@{}}Subspace dimension \\ (FE space dim. 3969)\end{tabular}\\
    \hline
    4 ($2\times 2$) & direct & 21 & 21 \\  
    4 ($2\times 2$)& intermediate & 21 & 21\\ 
    \hline
    8 ($2\times 4$)& direct & 77 & 83 \\  
    8 ($2\times 4$)& intermediate & 73 & 83\\
    \hline
    16 ($4\times 4$)& direct & 178 & 208 \\  
    16 ($4\times 4$)& intermediate & 181 & 217\\
    \hline
    \end{tabular}
    \caption{The dimensions of two kinds of compound Krylov subspaces of the order $5$. The cut-off tolerance for singular values is $\delta=10^{-7}$.}
    \label{table:dimensions}
    \end{center}
\end{table}

\begin{figure}
     \centering
     \begin{subfigure}[b]{0.45\textwidth}
         \centering
         \includegraphics[width=\textwidth]{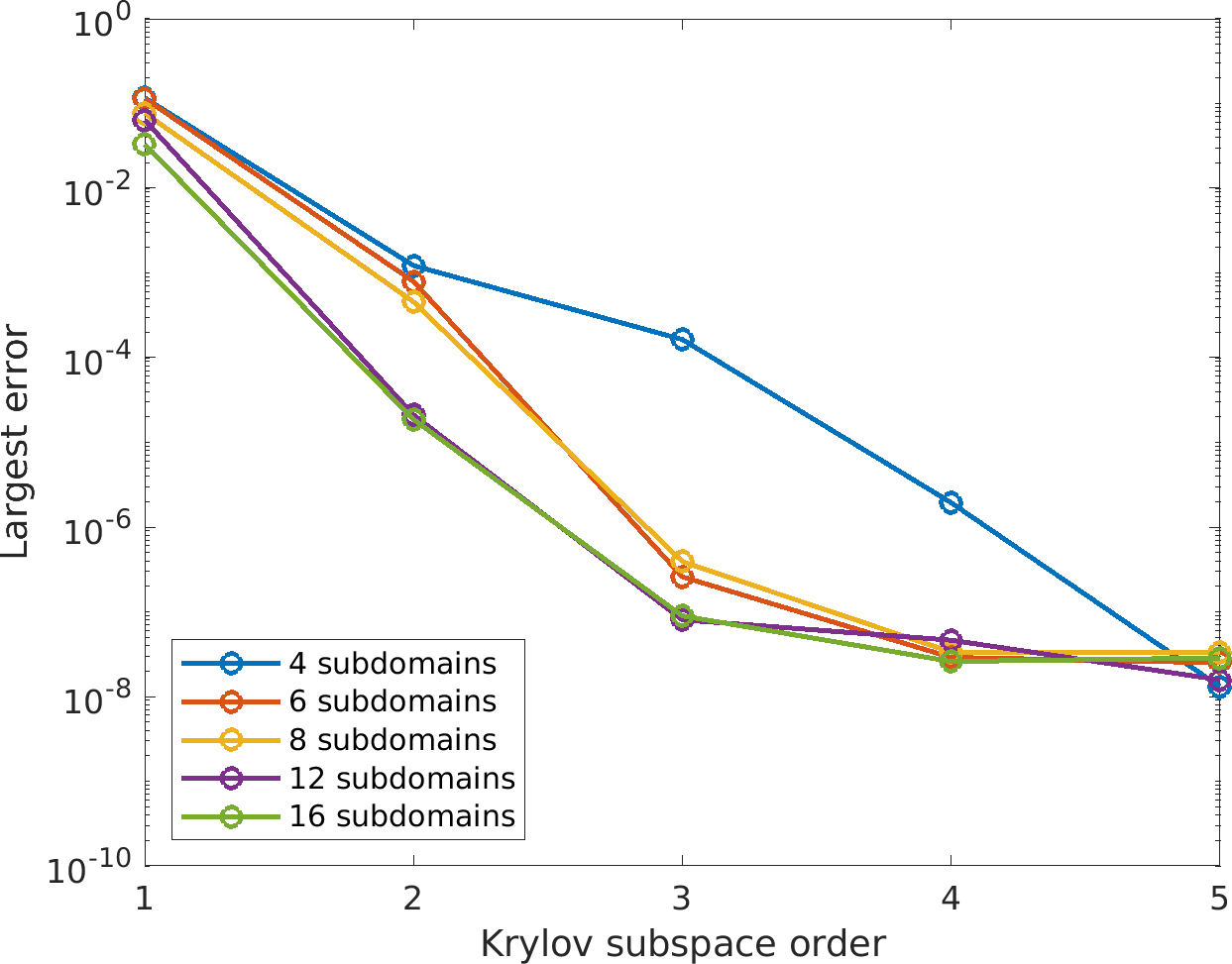}
     \end{subfigure}
     \hfill
     \begin{subfigure}[b]{0.45\textwidth}
         \centering
         \includegraphics[width=\textwidth]{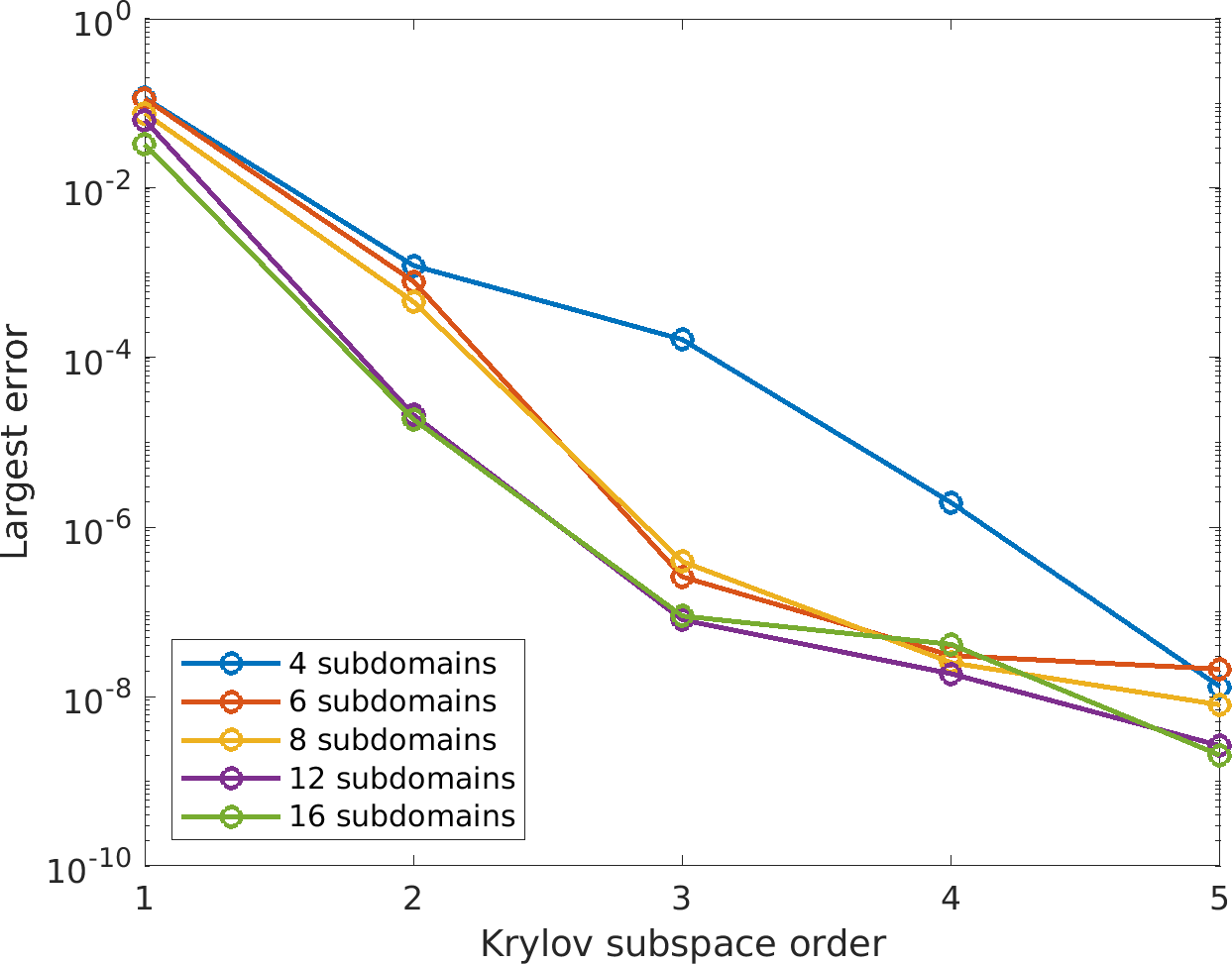}
     \end{subfigure}
     \caption{The largest errors in $A(\vec{\sigma})$-norm for $100$ randomly drawn vectors $\vec{\sigma}$ for the direct method (left) and the intermediate method (right) for different amounts of subdomains with $\delta=10^{-7}$.}
     \label{fig:errors_checkerboard}
\end{figure}

\begin{figure}
     \centering
     \begin{subfigure}[b]{0.45\textwidth}
         \centering
         \includegraphics[width=\textwidth]{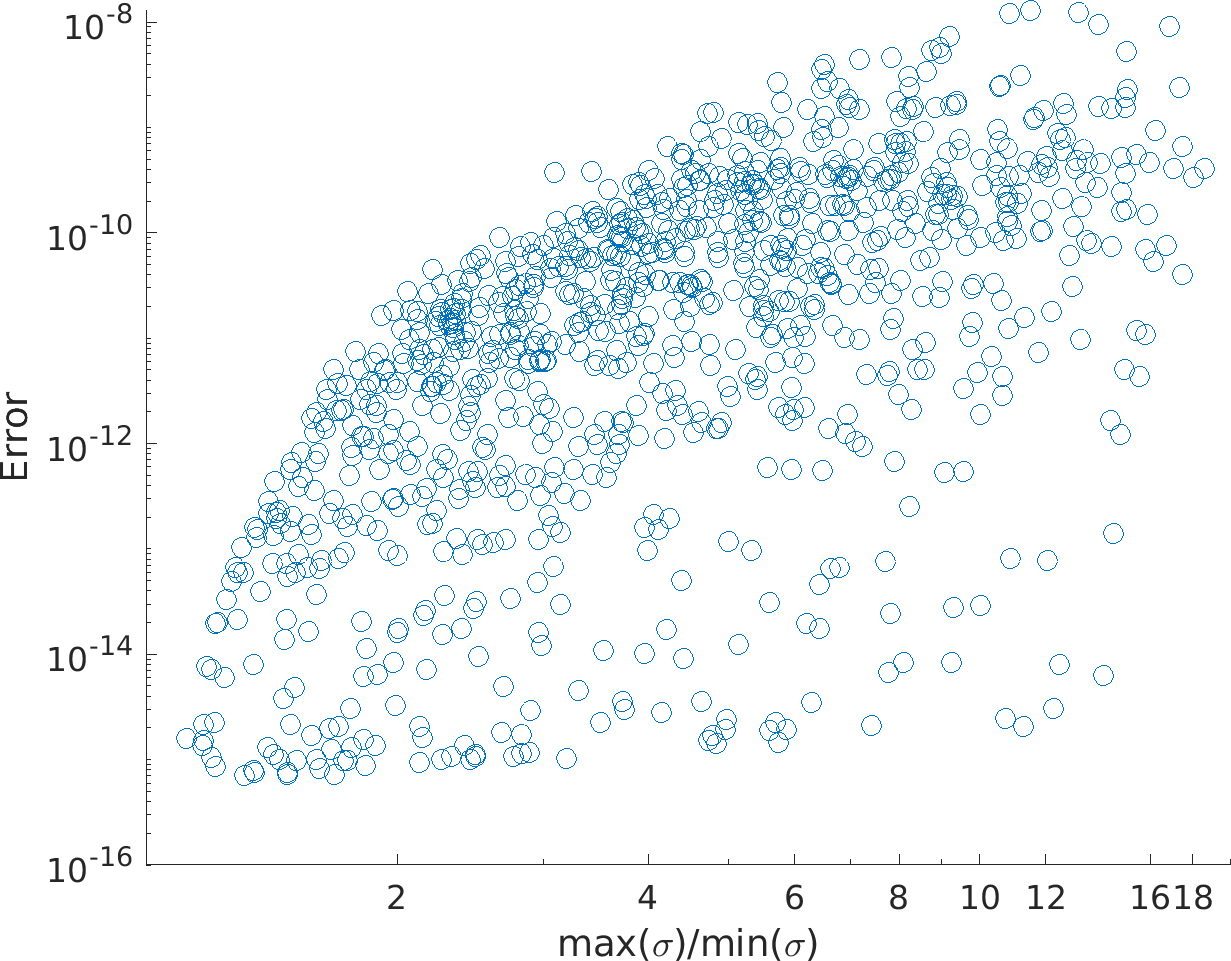}
         \caption{$2\times 2$-subdomains}
     \end{subfigure}
     \hfill
     \begin{subfigure}[b]{0.45\textwidth}
         \centering
         \includegraphics[width=\textwidth]{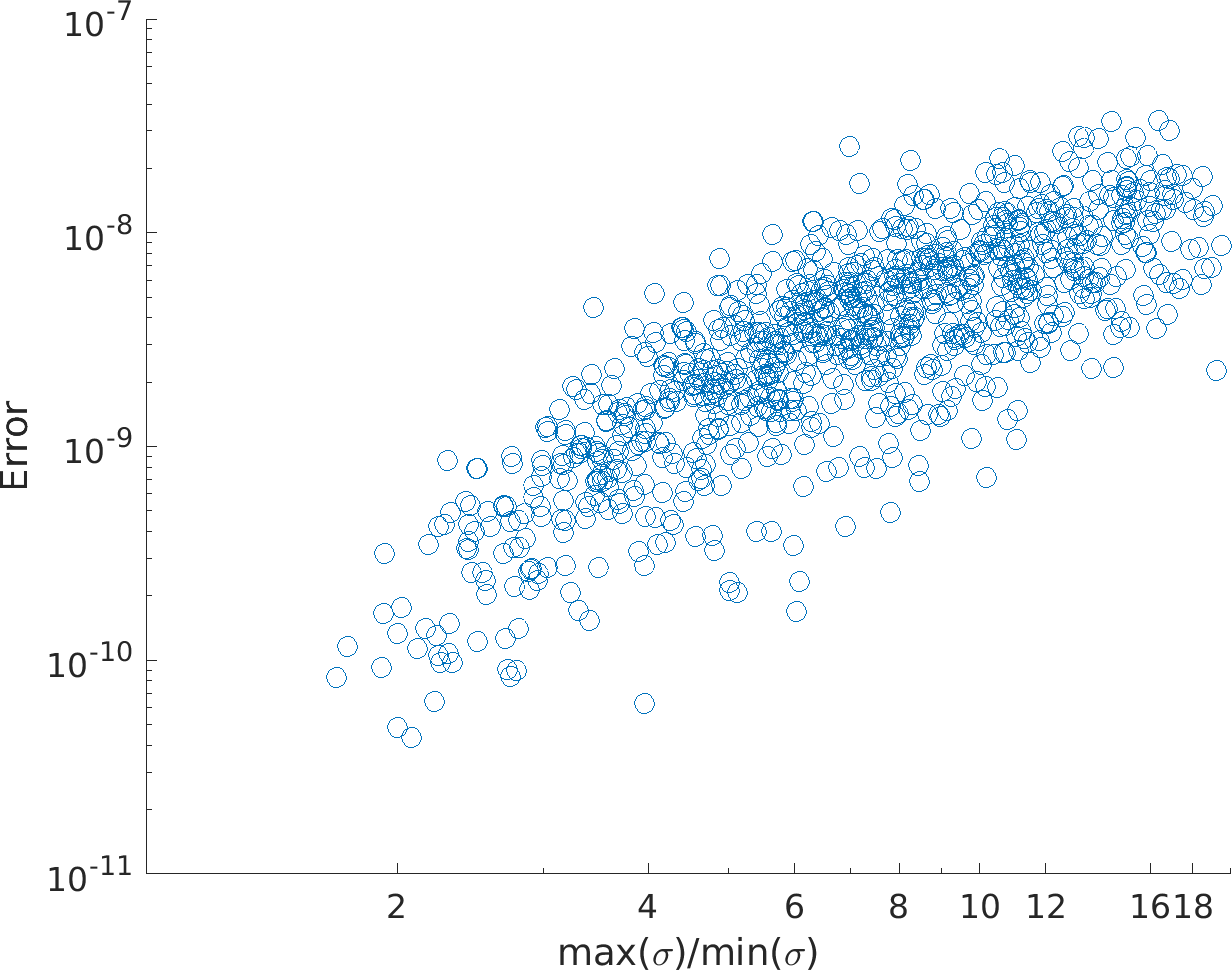}
         \caption{$2\times 4$-subdomains}
     \end{subfigure}
     \hfill
        \caption{Scatter plots of errors corresponding to randomly drawn vectors $\vec{\sigma}$. The subspace $CK^1_5(A,\vec{b})$ is used.}
        \label{fig:scatter}
\end{figure}

\begin{figure}
     \centering
     \begin{subfigure}[b]{0.3\textwidth}
         \centering
         \includegraphics[width=\textwidth]{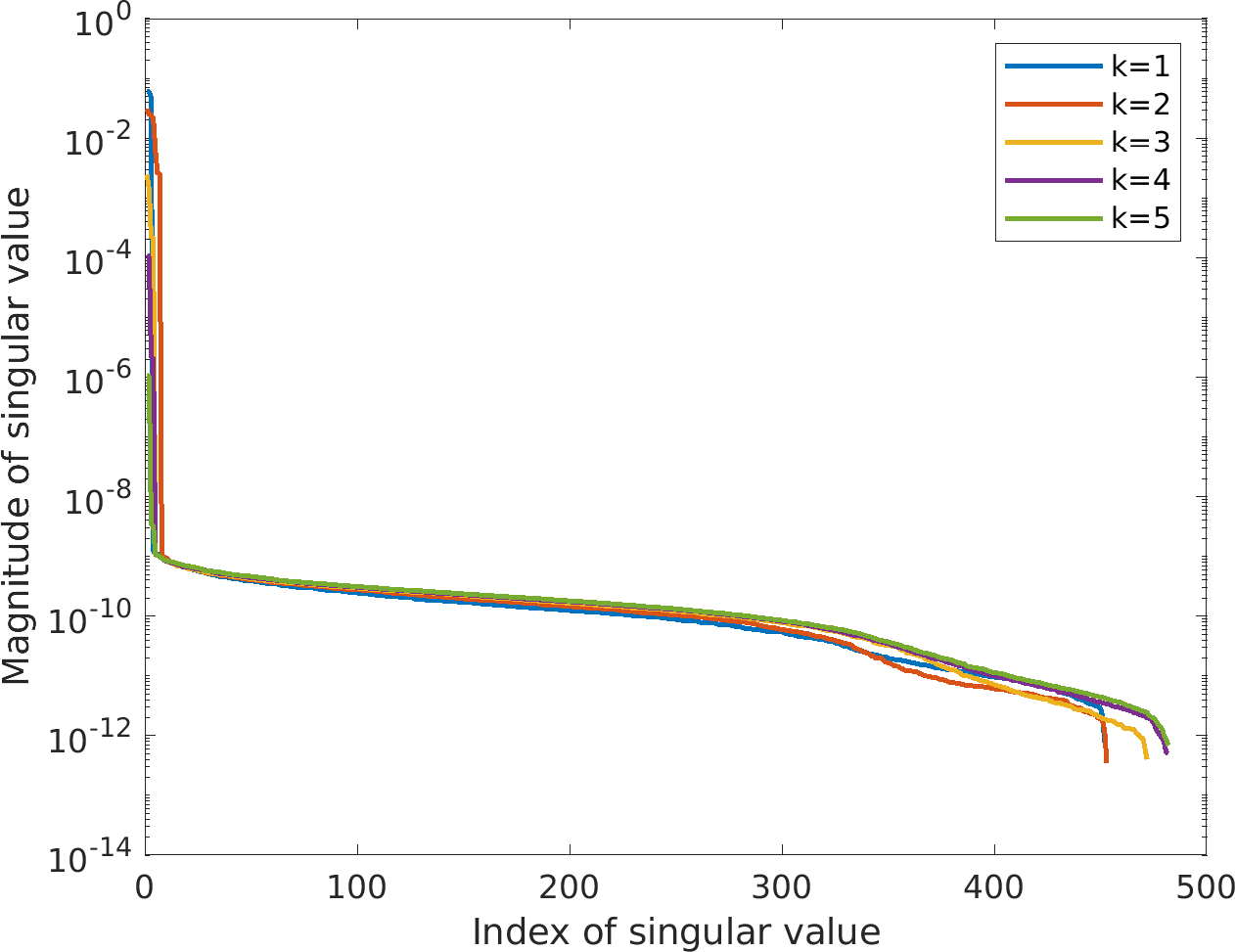}
         \caption{4 subdomains}
     \end{subfigure}
     \hfill
     \begin{subfigure}[b]{0.3\textwidth}
         \centering
         \includegraphics[width=\textwidth]{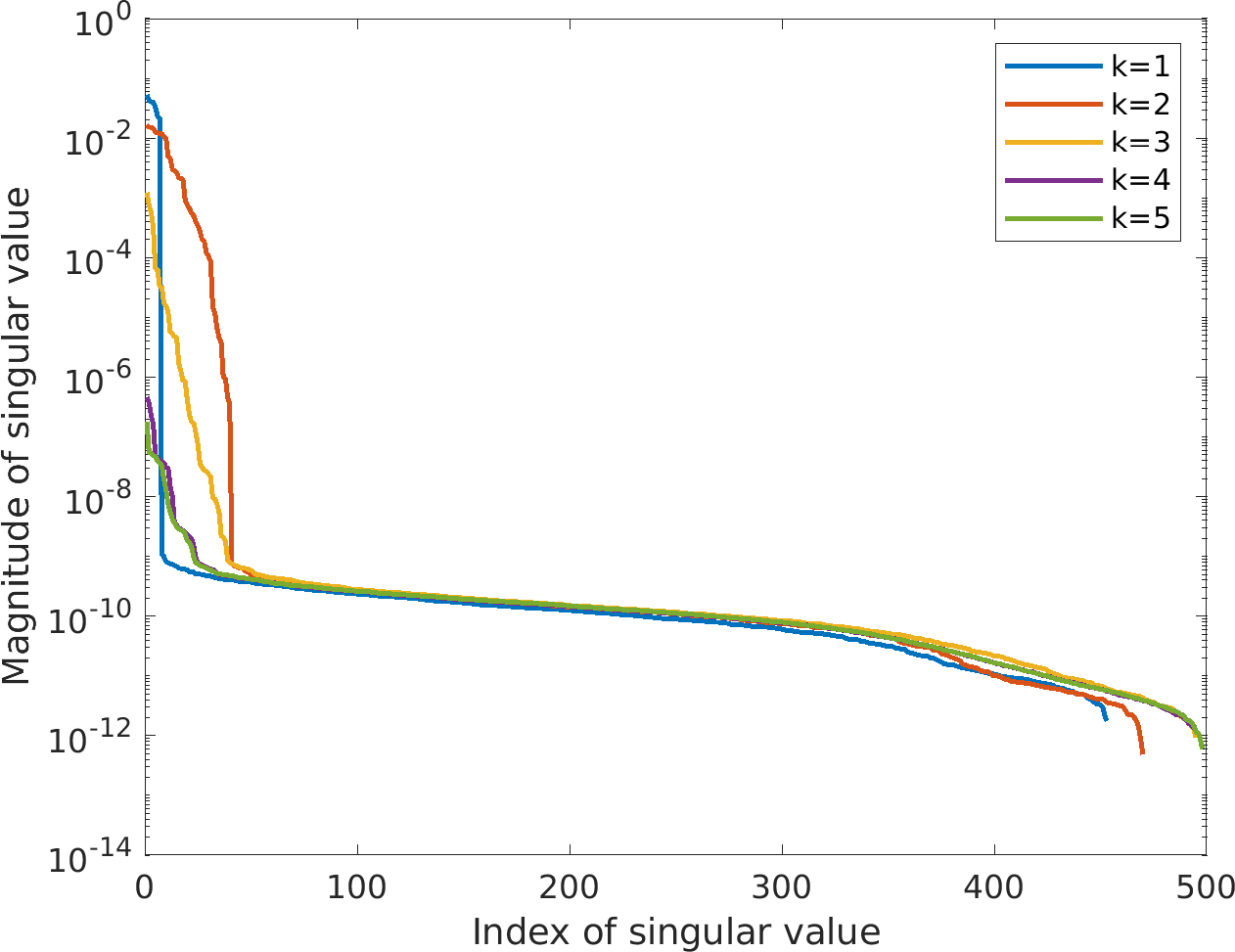}
         \caption{8 subdomains}
     \end{subfigure}
     \hfill
     \begin{subfigure}[b]{0.3\textwidth}
         \centering
         \includegraphics[width=\textwidth]{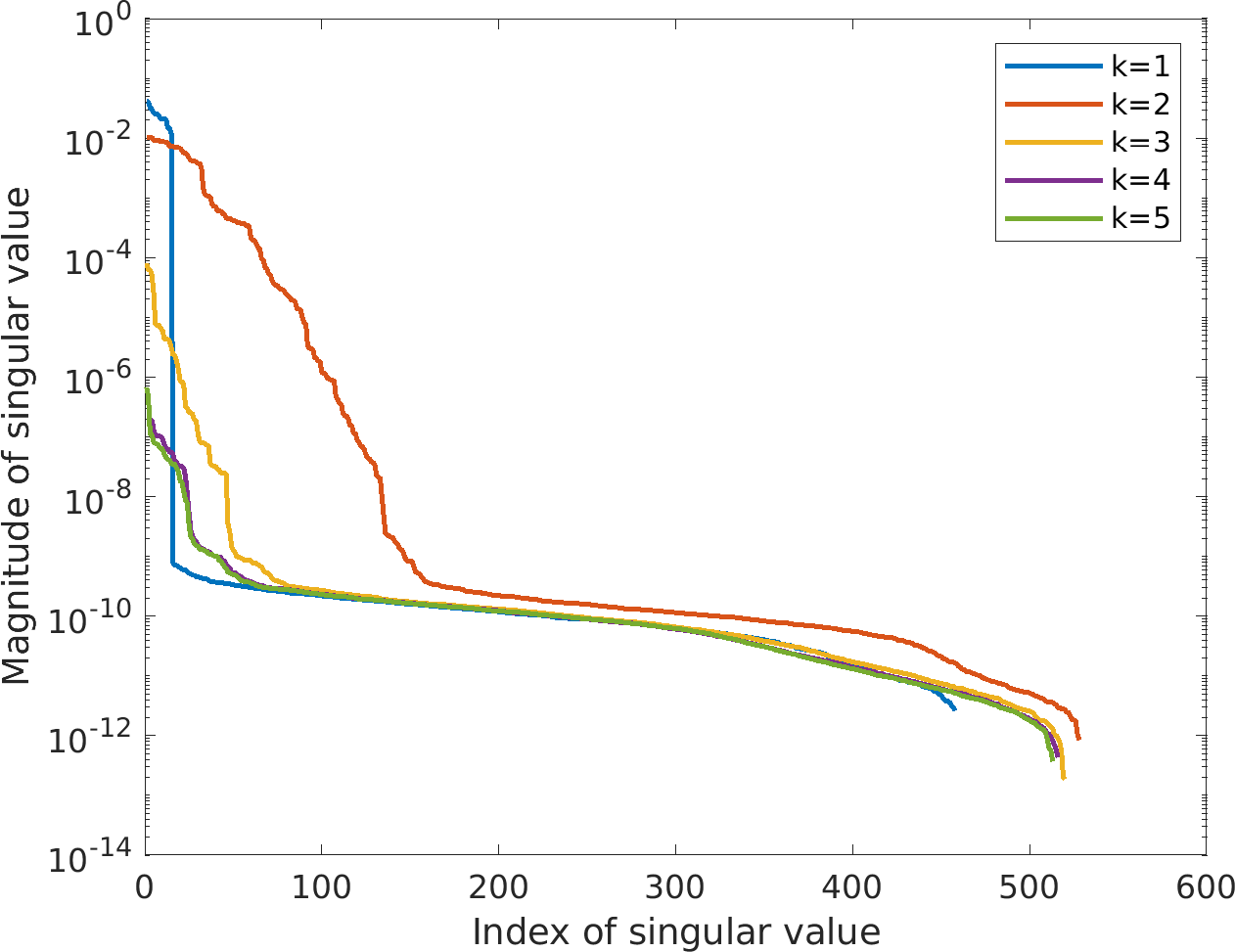}
         \caption{16 subdomains}
     \end{subfigure}
        \caption{The singular values for the projected linearisation matrices for different orders $j$.}
        \label{fig:singular_values}
\end{figure}

The singular values of the projected linearisation matrices for different values of $k$ are plotted in Fig. \ref{fig:singular_values}. Singular values seem to decay exponentially and many of them are numerically zero. As the $CK$ subspaces of the first kind are constructed from the singular vectors corresponding to the largest singular values of the projected linearisation matrix, using the $\delta$-accurate low rank approximation significantly lowers the dimension of this subspace. 

The effect of the singular value cut-off $\delta$ to error is studied in Fig. \ref{fig:cutoff}, where the error is plotted as a function of $\sigma_r$, the smallest singular value satisfying $\sigma_r \geq \delta$. The relationship is quite linear, even for small values of $\delta$. This is actually better that could be expected based on Theorems~\ref{thm:error1} and \ref{thm:error2}, which state that the improvement in the error from reducing the cut-off tolerance should stop at some point due to the error associated with the Krylov subspace itself.

\subsection{Deformation of geometry}
Next, we consider the  example described in Sec. \ref{sec:hole}. The loading is chosen as $f=1$ and the cut-off tolerance as $\delta_k=\delta=10^{-7}$. The error as a function of the translation $l$ is plotted in Fig. \ref{fig:hole_l} for the spaces $CK^1_5$ and $CK^2_5$. As expected by the condition number estimate \eqref{eq:hole_kappa}, the error grows with the translation. For the  largest plotted value of $l$, $l=0.3$, the hole is very close to the boundary of the rectangle so it's intuitive that the error is largest in that case.

The effect of the compound Krylov subspace order $j$ to the errors for the value $l=0.1$ is shown in Fig. \ref{fig:hole_order}. The improvement is exponential with both methods. In this example we do not yet see the improvement flattening out due to the cut-off tolerance $\delta$. The dimension of the FE space used in the tests was 1088. The dimension of the subspaces $CK^1_5$ and $CK^2_5$ were 240 and 238, respectively. The singular values of the projected linearisation matrices corresponding to $CK^1_j$ are shown in Fig. \ref{fig:hole_svs}. 
\begin{figure}
    \centering
    \begin{minipage}[t]{.45\textwidth}
        \centering
        \includegraphics[width=\linewidth]{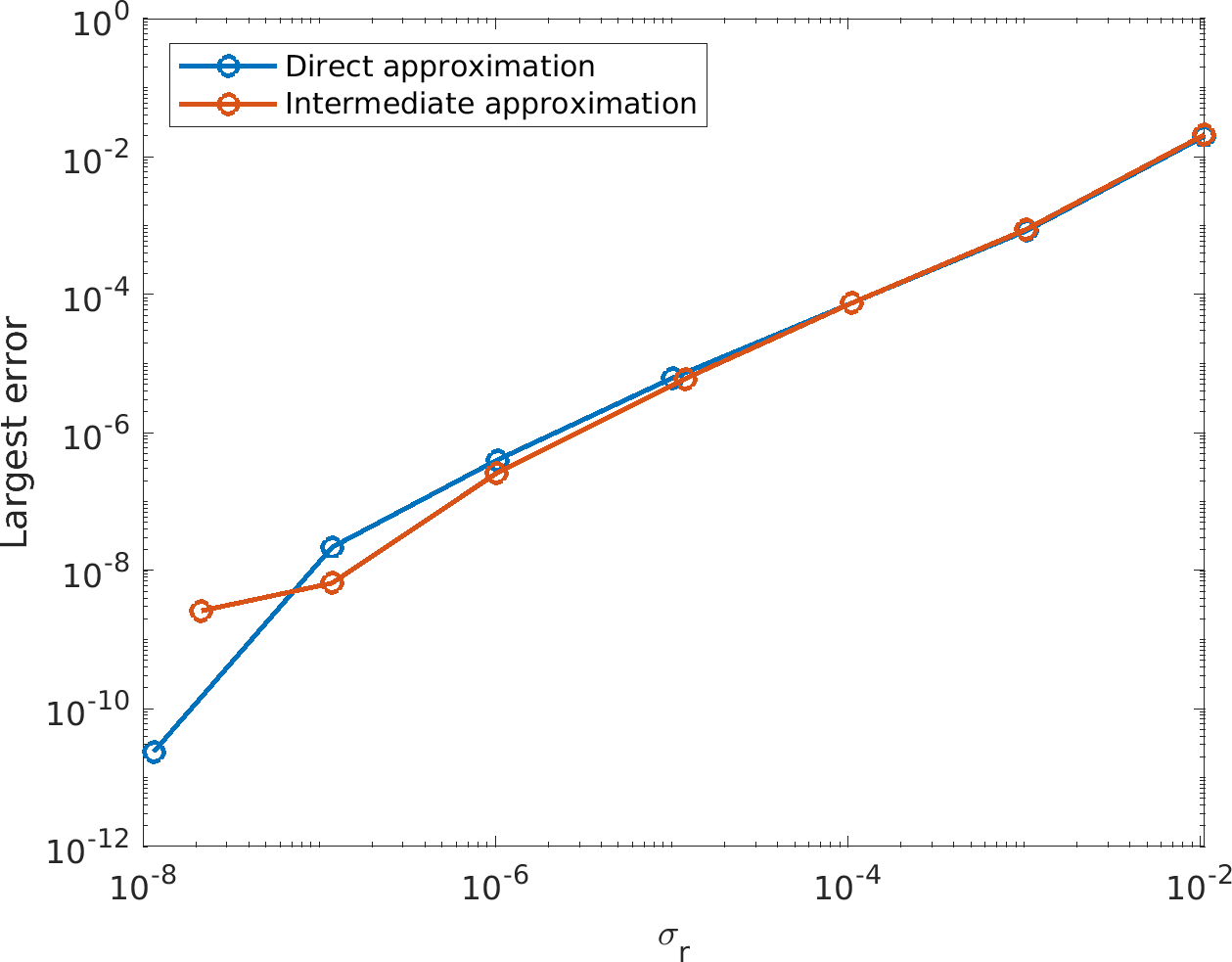}
        \caption{The effect of the singular value cut-off on the $A(\vec{\sigma})$-norm errors using a $2\times 4$ grid of the material parameter example.}
        \label{fig:cutoff}
    \end{minipage}
    \hfill
    \begin{minipage}[t]{.45\textwidth}
        \centering
        \includegraphics[width=\linewidth]{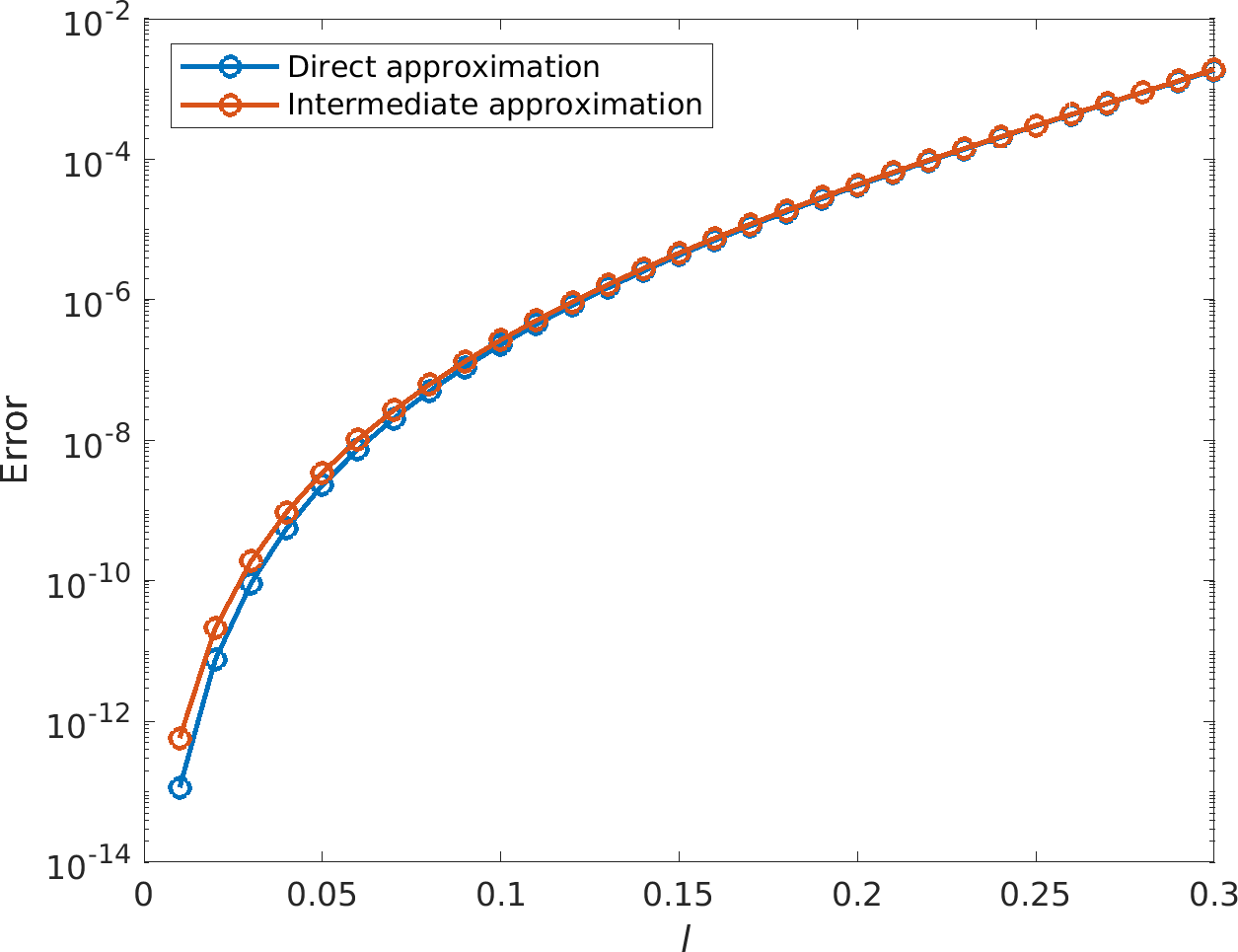}
        \caption{The effect of the translation length $l$ on the $A(\vec{\sigma})$-norm errors of the deformation of geometry example.}
        \label{fig:hole_l}
    \end{minipage}
\end{figure}

\begin{figure}
    \centering
    \begin{minipage}[t]{.45\textwidth}
        \centering
        \includegraphics[width=\linewidth]{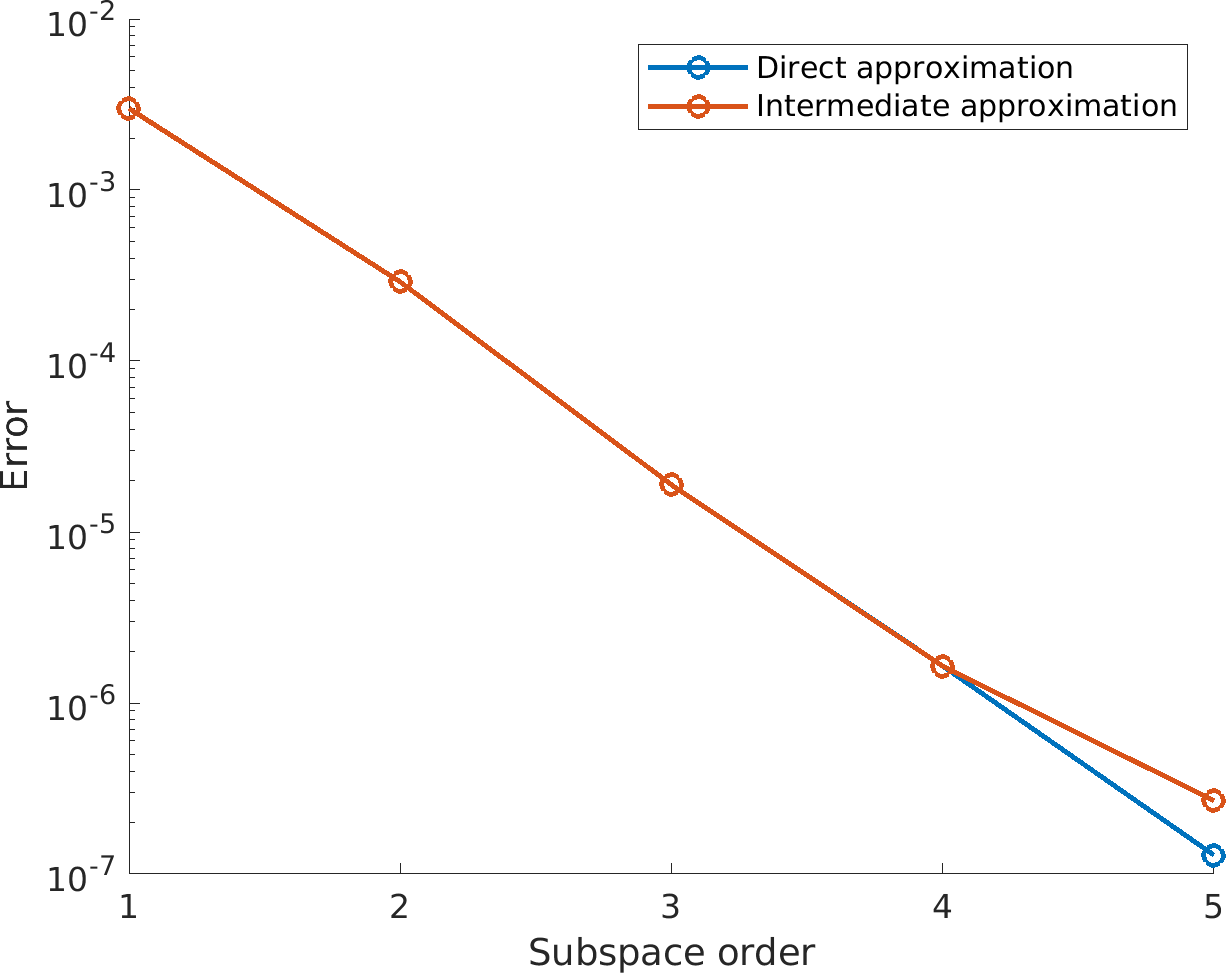}
        \caption{The effect of the Krylov subspace order $j$ to the $A(\vec{\sigma})$-norm errors of the deformation example. The parameter $l=0.1$.}
        \label{fig:hole_order}
    \end{minipage}
    \hfill
    \begin{minipage}[t]{.45\textwidth}
        \centering
        \includegraphics[width=\linewidth]{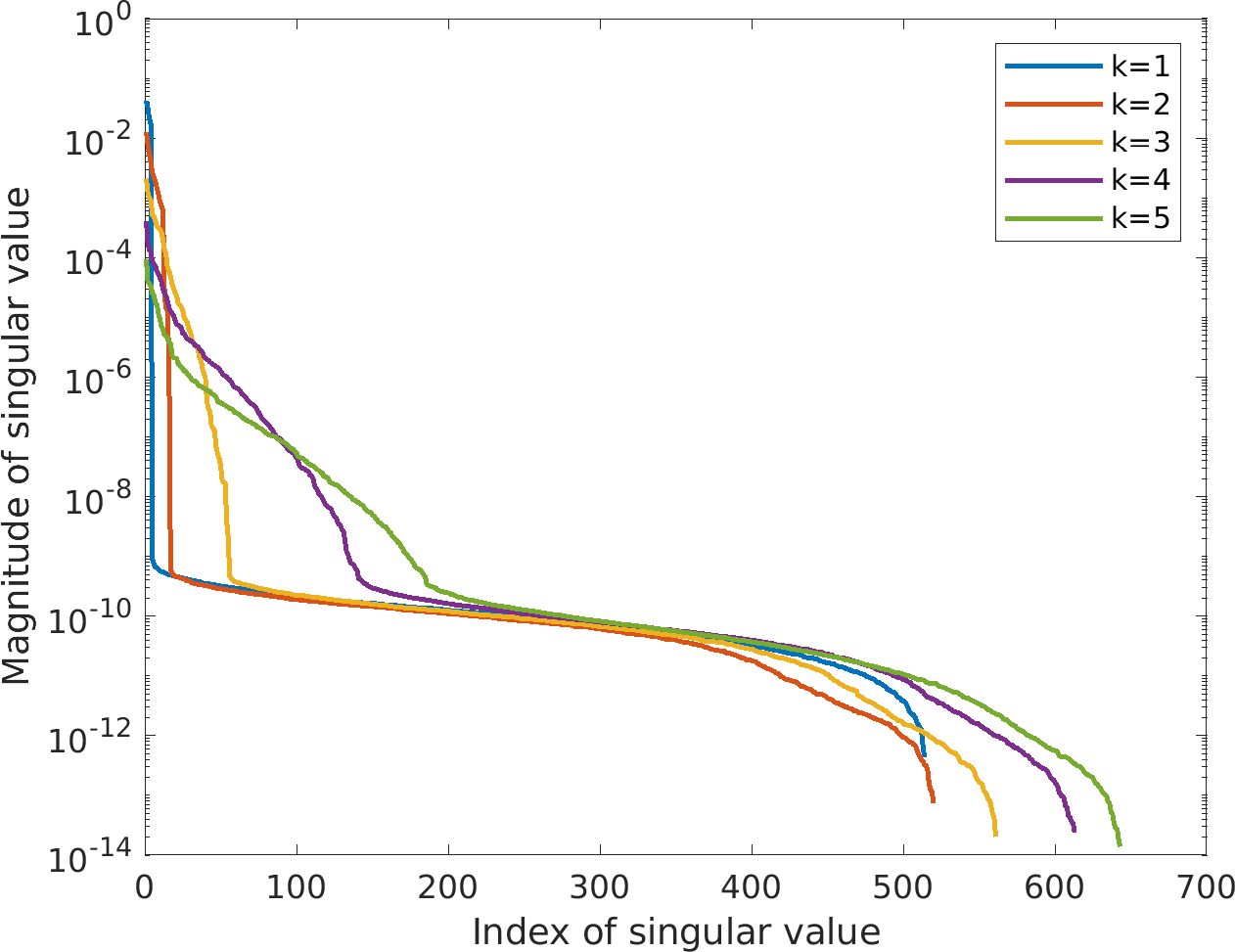}
        \caption{The singular values for the projected linearisation matrices for $CK^1$ with different numbers of iterative rounds in the deformation example.}
        \label{fig:hole_svs}
    \end{minipage}
\end{figure}

\section{Conclusions} \label{sec:con} This article presents a reduced basis method for the solution of the parametric linear system \eqref{eq:linsys}. The proposed method utilises the linearity of the coefficient function $A(\vec{\sigma})$ to construct a basis for associated compound Krylov subspace that contains standard Krylov subspaces for each $\vec{\sigma} \in S$. The basis is computed using linearisation given in Lemma~\ref{lemma:Lprop} and Lemma~\ref{lemma:normal_form}. Two approximate variants utilising low-rank approximations of the linearisation matrices are proposed, see Definitions~\ref{process:svd} and \ref{process:interspace}. The error due to the low-rank approximation is bounded in Theorems~\ref{thm:error1} and \ref{thm:error2}. Practical algorithms for both variants are given in Section~\ref{sec:practical}. Numerical examples illustrate the presented analysis. 

Our numerical examples indicate that the method preforms even better than what the mathematical analysis states. The dimension of the CK subspaces is small due to exponential decay of singular values of the projected linearisation matrices. This is not proven, and presents a challenging topic for future research. Also, the current method does not utilise the sparsity of finite element matrices, that is another topic for future work. 

\bibliographystyle{plain}
\bibliography{master}

\end{document}